\documentclass{article}
\usepackage{hyperref}
\usepackage{graphicx}
\usepackage[all]{xy} 
\usepackage{amsmath}
\usepackage{amsthm}
\usepackage{mathdots}
\usepackage{yhmath}
\usepackage{array}  
\usepackage{amssymb}
\usepackage{enumerate}
\usepackage{calc}
\usepackage{txfonts}
\usepackage{textcomp}
\usepackage{tabularx} 
\usepackage[table]{xcolor}
\usepackage{hhline}
\usepackage{lscape}
\usepackage{rotating}
\usepackage{tensor}
\usepackage{arydshln} 
\usepackage{bbm} 
\usepackage{appendix} 
\usepackage{placeins} 
\usepackage[T1]{fontenc} 
\usepackage{titletoc}
\usepackage{pgf}
\usepackage{tikz}
\usetikzlibrary{matrix,arrows}
\usetikzlibrary{patterns}
\usetikzlibrary{shapes,decorations,shadows}
\usepackage{young}
\usepackage[enableskew]{youngtab}

\DeclareMathOperator{\End}{End} 
\DeclareMathOperator{\Ext}{Ext}

\DeclareMathOperator{\Hom}{Hom} 
\DeclareMathOperator{\rad}{rad} 
\DeclareMathOperator{\corank}{corank}

\DeclareMathOperator{\sh}{sh}

\DeclareMathOperator{\rep}{rep}

\DeclareMathOperator{\GL}{GL}

\definecolor{lightgreen}{rgb}{0.8,1,0.8}
\definecolor{dlightgreen}{rgb}{0.6,1,0.6}
\definecolor{lightblue}{rgb}{0.64,0.83,0.93}
\definecolor{dlightblue}{rgb}{0.39,0.58,0.93}
\definecolor{lightviolet}{rgb}{0.85,0.44,0.84}
\definecolor{dlightviolet}{rgb}{0.7,0.23,0.94}
\definecolor{grey}{rgb}{0.8,0.8,0.8}

\newcommand*{\punkte}{\dots\unkern}
\newcolumntype{C}[1]{>{\centering\arraybackslash}p{#1}}
\newcommand{\A}{\mathcal{A}} 
\newcommand{\B}{\mathcal{B}}

\newcommand{\Q}{\mathcal{Q}} 
\newcommand{\Orb}{\mathcal{O}} 
\newcommand{\N}{\mathcal{N}}

\newcommand{\dimv}{\underline{\dim}}
 
\newcommand{\df}{\underline{d}}

\newcommand{\CC}{\mathcal{C}}

\newtheorem{theorem}{Theorem}[section]
\newtheorem{lemma}[theorem]{Lemma}
\newtheorem{definition}[theorem]{Definition}
\newtheorem{proposition}[theorem]{Proposition}
\newtheorem{corollary}[theorem]{Corollary}
\newtheorem{remark}[theorem]{Remark}

\newtheorem{example}[theorem]{Example}

\newtheorem*{theorem*}{Theorem}
\begin{document}
\parindent0pt
\title{\bf Staircase algebras and graded nilpotent pairs}

\author{Magdalena Boos\\Ruhr-Universit\"at Bochum\\ Faculty of Mathematics\\  D - 44780 Bochum\\ Magdalena.Boos-math@ruhr-uni-bochum.de}
\date{}
\maketitle

\begin{abstract}
We consider a class of finite-dimensional algebras, the so-called \textquotedblleft Staircase algebras\textquotedblright~ parametrized by Young diagrams.  We develop a complete classification of representation types of these algebras and look into finite, tame (concealed) and wild cases in more detail. Our results are translated to the setup of graded nilpotent pairs for which we prove certain finiteness conditions.\\[1ex]
\end{abstract}

\section{Introduction}\label{intro}
Graded nilpotent pairs naturally appear as a generalization of principal nilpotent pairs introduced by Ginzburg \cite{Gi} and studied by Panyushev \cite{Pan3}, Yu \cite{Yu} and others.\\[1ex]
 Such pair is given as an element $(\varphi, \psi)$ of the nilpotent commuting variety of a fixed bigraded vector space $V$, such that both nilpotent operators are compatible with the bigrading in a natural commuting way: $\varphi$ respects the bigrading "horizontally" and $\psi$ respects the bigrading "vertically". The non-zero components of the bigrading of $V$ induce a partition $\lambda$ and the graded nilpotent pair is called $\lambda$-shaped, thus. \\[1ex] 
One of our main results gives an answer to a standard Lie-theoretic question: Are there only finitely many $\lambda$-shaped graded nilpotent pairs up to base change by a Levi respecting the grading?
\begin{theorem*}[\ref{gradedLambda}]
The number of $\lambda$-shaped graded nilpotent pairs is finite (modulo base change in the homogeneous components) if and only if 
\begin{enumerate}
 \item $\lambda\in\{(n),~ (1^k,n-k),~ (2,n-2),~ (1^{n-4},2^2)\}$ for some $k\leq n$,
\item  $n\leq 8$ and $\lambda\notin\{(1,3,4),~ (2,3^2),~ (1,2^2,3),~ (1^2,2,4)\}$.
\end{enumerate}
\end{theorem*}
Furthermore, we obtain (in)finiteness conditions for graded nilpotent pairs (modulo Levi-base change) on a fixed bigraded vector space $V$ in Lemma \ref{gradedTame} and Lemma \ref{gradedInfinite}.\\[1ex]
In order to obtain our results on graded nilpotent pairs, we
 introduce a class of finite-dimensional algebras, the so-called staircase algebras in Section \ref{SectStaircase}. We show that every graded nilpotent pair can be considered as a representation of such a staircase algebra and approach this class of algebras from the angle of  representation theory of finite-dimensional algebras (for example with Auslander-Reiten techniques). Our results are translated to graded nilpotent pairs in Section \ref{SectGradedNilp}.\\[1ex]
Staircase algebras, denoted by $\A(\lambda)$, are quite interesting themselves - they are parametrized by Young diagrams $Y(\lambda)$ (or, equivalently, partitions $\lambda$) and are defined by quivers with (exclusively) commutativity relations. Our main aim concerning staircase algebras is to classify their representation types completely in Section \ref{SectRepTypes}. 
\begin{theorem*}[\ref{reptype}]
A staircase algebra $\A(\lambda)$ is 
\begin{itemize}
\item representation-finite if and only if the conditions of Theorem \ref{gradedLambda} hold true.
\item tame concealed if and only if $\lambda$ comes up in the following list:\\ $(3,6)$, $(1,2,6)$, $(1,3,4)$, $(2^2,5)$, $(1^2,2,4)$, $(1,2^2,3)$, $(1^3,3^2)$, $(1^3,2^3)$, $(1^4,2,3)$.
\item tame, but not tame concealed if and only if $\lambda$ comes up in the following list:\\ $(4,5)$, $(5^2)$, $(1,4^2)$, $(2,3^2)$, $(3^3)$, $(2^3,3)$, $(1,2^4)$, $(2^5)$.
\end{itemize} 
 Otherwise, $\A(\lambda)$ is of wild representation type. 
\end{theorem*}
The orbit type of each staircase algebra is classified in Lemma \ref{orbittype} and we prove a correlation between representation types and orbit types for the class of staircase algebras.  We obtain a complete hierarchy of algebras which makes clear the transitions between representation types and is visualized in Appendix \ref{hierarchy}.\\[1ex]
In Section \ref{SectCasestudy}, the finite, tame (concealed) and minimal wild cases are examined in more detail. For example, all staircase algebras of finite type and most of the tame ones are tilted and (except for infinite families of representation-finite algebras) all Auslander-Reiten quivers of the former are attached in Appendix \ref{appendix}.\\[1ex]
For each tame case, minimal nullroots which admit infinitely many isomorphism classes of representations are provided initiating the study of finiteness criteria for nilpotent graded pairs of a fixed bi-graded vector space.
For each case of wild representation type, we construct a $2$-parameter family of pairwise non-isomorphic representations.~\\[2ex]
{\bf Acknowledgments:} I am grateful to K. Bongartz and M. Reineke for various helpful discussions and ideas. I would like to thank M. Bulois and J. K\"ulshammer for helpful remarks and suggestions. Furthermore, I am thankful to  O. Kerner and L. Unger for helping me with finding and getting to know literature and known results. 
\section{Theoretical background}\label{SectTheory}
Let $K=\overline{K}$ be an algebraically closed field and let $\GL_n\coloneqq\GL_n(K)$ be the general linear group for a fixed integer $n\in\textbf{N}$ regarded as an affine variety. We begin by defining the notion of a graded nilpotent pair before including some facts about the representation theory of finite-dimensional algebras. We refer to \cite{ASS} for a thorough treatment of the latter.
\subsection{Graded nilpotent pairs}\label{gradNilpTheory}
Let $V= \bigoplus_{s,t\in\mathbf{Z}_{\geq 1}} V_{s,t}$ be an $N$-dimensional bigraded $K$-vector space; we formally set $V_{x,y}:=0$ for $(x,y)\notin \mathbf{Z}_{\geq 1}^2$.\\[1ex]
 Denote by $\N(V)$ the \textit{nilpotent cone} of nilpotent operators on $V$. The \textit{nilpotent commuting variety} of $V$ is defined by
\[\CC(V):=\{(\varphi, \psi)\in \N(V)\times \N(V)\mid [\varphi, \psi]=0\},\]
its elements are called \textit{commuting nilpotent pairs}.\\[1ex]
Such pair is called \textit{graded nilpotent pair}, if $\varphi$  restrict to each $V_{s,t}$ "horizontally" via 
\[\varphi_{s,t}:=\varphi\mid_{V_{s,t}}: V_{s,t} \rightarrow V_{s-1,t}\] and $\psi$ restrict to each $V_{s,t}$ "vertically" via
\[\psi_{s,t}:=\psi\mid_{V_{s,t}}: V_{s,t} \rightarrow V_{s,t-1}.\]
Note that the study of this setup is quite natural regarding the context of principal nilpotent pairs \cite{Gi}.\\[1ex]    
We define the \textit{shape} of $V$ by \[\sh(V):=\{(s,t)\mid \exists p,q\in\mathbf{Z}_{\geq 1}, p\geq s, q\geq t: V_{p,q}\neq 0 \}\]
which defines a Young diagram corresponding to a partition $\lambda(V)$. In more detail, the latter is given by $\lambda(V)_{i}= \sharp\{(s,h(V)-i)\in\sh(V) \mid s\in \mathbf{Z}_{\geq 1}\} $ if we define $h(V):= \max\{t\mid V_{1,t}\neq 0\} +1$.\\[1ex]
 In this case,  $(\varphi,\psi)$ is called a \textit{$\lambda$-graded nilpotent pair}. 
Note that these definitions depend on a fixed chosen grading of $V$, but not on the nilpotent pairs compatible with it. 
\begin{example}\label{gradedExample}
Let $V:= \bigoplus_{s,t\in\mathbf{Z}_{\geq 1}} V_{s,t}$ be a bigraded $K$-vector space, such that $V_{3,1}= V_{1,3}=V_{4,1}=K$, $V_{2,1}= V_{1,2}= K^2$, $V_{2,2}=K^3$ and $V_{s,t}=0$, otherwise. Let $(\varphi, \psi)$ be a graded nilpotent pair on $V$. Then we can illustrate the latter by
\begin{center}\small\begin{tikzpicture}
\matrix (m) [matrix of math nodes, row sep=1.81em,
column sep=1.8em, text height=0.9ex, text depth=0.1ex]
{ V_{4,3} &V_{3,3} & V_{2,3} &  V_{1,3} & & 0 & 0 & 0 &  K \\
V_{4,2} & V_{3,2} & V_{2,2} & V_{1,2} & = & 0 & 0 & K^3 & K^2 \\
V_{4,1} & V_{3,1} & V_{2,1} & V_{1,1} & & K & K & K^2 & 0 \\
 };
\path[->]
(m-1-4) edge  (m-2-4)
(m-1-3) edge  (m-2-3)
(m-1-2) edge  (m-2-2)
(m-1-1) edge  (m-2-1)
(m-1-1) edge  (m-1-2)
(m-1-2) edge  (m-1-3)
(m-1-3) edge  (m-1-4)
(m-2-1) edge  (m-2-2)
(m-2-2) edge  (m-2-3)
(m-2-3) edge  (m-2-4)
(m-2-1) edge  (m-3-1)
(m-2-2) edge  (m-3-2)
(m-2-3) edge  (m-3-3)
(m-2-4) edge  (m-3-4)
(m-3-1) edge  (m-3-2)
(m-3-2) edge  (m-3-3)
(m-3-3) edge  (m-3-4)
(m-1-1) edge[-,dotted]  (m-2-2)
(m-1-2) edge[-,dotted]  (m-2-3)
(m-1-3) edge[-,dotted]  (m-2-4)
(m-2-1) edge[-,dotted]  (m-3-2)
(m-2-2) edge[-,dotted]  (m-3-3)
(m-2-3) edge[-,dotted]  (m-3-4)
(m-2-8) edge node[above]{$\varphi_{2,2}$} (m-2-9)
(m-2-8) edge node[left]{$\psi_{2,2}$} (m-3-8)
(m-1-9) edge node[right]{$\psi_{1,3}$} (m-2-9)
(m-2-9) edge  (m-3-9)
(m-3-6) edge node[below]{$\varphi_{4,1}$}  (m-3-7)
(m-3-7) edge node[below]{$\varphi_{3,1}$}  (m-3-8)
(m-3-8) edge  (m-3-9)
(m-2-8) edge[-,dotted]  (m-3-9)
;\end{tikzpicture}\end{center} 

The Young diagram is accentuated on the right hand side, it is given by
\[Y(\lambda) = \begin{tabular}{|c|c|c}\cline{1-1}
(4,1)&\multicolumn{2}{c}{}\\ \cline{1-1}
(3,1)&\multicolumn{2}{c}{}\\ \cline{1-2}
(2,1)&(2,2)&\\ \hline
(1,1)&\multicolumn{2}{c|}{(1,2) \hspace{0.12cm}\vline ~~(1,3)} \\ \hline
\end{tabular}.\]
Thus, $\lambda(V)=(1,1,2,3)$ and $(\varphi,\psi)$ is a $\lambda(V)$-graded nilpotent pair.
\end{example}

\subsection{Basics of Representation Theory}\label{reptheory}
A \textit{finite quiver} $\Q$ is a directed graph $\Q=(\Q_0,\Q_1,s,t)$, given by a finite set of \textit{vertices} $\Q_0$, a finite set of \textit{arrows} $\Q_1$ and two maps  $s,t:\Q_1\rightarrow \Q_0$ defining the source $s(\alpha)$ and the target $t(\alpha)$ of an arrow $\alpha$. 
A \textit{path} is a sequence of arrows $\omega=\alpha_s\punkte\alpha_1$, such that $t(\alpha_{k})=s(\alpha_{k+1})$ for all $k\in\{1,\punkte,s-1\}$.\\[1ex]
 We define the \textit{path algebra} $K\Q$ as the $K$-vector space with a basis consisting of all paths in $\Q$, formally included is a path $\varepsilon_i$ of length zero for each $i\in \Q_0$ starting and ending in $i$. The multiplication  is defined by concatenation of paths, if possible, and equals $0$, otherwise.\\[1ex]
Let us define the \textit{arrow ideal} $R_{\Q}$ of $K\Q$ as the (two-sided) ideal which is generated (as an ideal) by all arrows in $\Q$. In particular, the arrow ideal of $\Q$ equals the radical of $K\Q$ if $\Q$ does not contain oriented cycles. An arbitrary ideal $I\subseteq K\Q$ is called \textit{admissible} if there exists an integer $s$, such that $R_{\Q}^s\subset I\subset R_{\Q}^2$. Given such admissible ideal $I$, we denote by $I(i,j)$ the set of paths in $I$ starting in $i$ and ending in $j$.\\[1ex]
Given such admissible ideal $I$, the path algebra of $\Q$, \textit{bound by} $I$ is defined as $\A:=K\Q/I$; it is a basic and finite-dimensional $K$-algebra \cite{ASS}; the elements of $I$ are the so-called \textit{relations} of $\A$. Then $\A$ is called \textit{triangular}, if $\Q$ does not contain an oriented cycle.\\[1ex]
Let $\rep_K\A$ be the abelian $K$-linear category of all finite-dimensional $\A$-representations which is equivalent to the category of  \textit{$K$-representations} of $\Q$, which are \textit{bound by $I$}, defined as follows: 
The objects are given by tuples 
$((M_i)_{i\in \Q_0},(M_\alpha)_{\alpha\in \Q_1})$, 
where the $M_i$ are $K$-vector spaces, and the $M_\alpha\colon M_{s(\alpha)}\rightarrow M_{t(\alpha)}$ are $K$-linear maps. For each path $\omega$ in $\Q$ as above, we denote $M_\omega=M_{\alpha_s}\cdot\punkte\cdot M_{\alpha_1}$ and ask a representation $M$ to fulfill $\sum_\omega\lambda_\omega M_\omega=0$ whenever $\sum_\omega\lambda_\omega\omega\in I$. A \textit{morphism of representations} $M=((M_i)_{i\in \Q_0},(M_\alpha)_{\alpha\in \Q_1})$ and
 \mbox{$M'=((M'_i)_{i\in \Q_0},(M'_\alpha)_{\alpha\in \Q_1})$} consists of a tuple of $K$-linear maps $(f_i\colon M_i\rightarrow M'_i)_{i\in \Q_0}$, such that $f_jM_\alpha=M'_\alpha f_i$ for every arrow $\alpha\colon i\rightarrow j$ in $\Q_1$. \\[1ex]
The \textit{dimension vector} of an $\A$-representation $M$  is defined by $\dimv M\in\mathbf{N}^{\Q_0}$; in more detail $(\dimv M)_{i}=\dim_K M_i$ for $i\in \Q_0$. 
Let us fix such a dimension vector $\df\in\mathbf{N}^{\Q_0}$ and denote by $\rep_K\A(\df)$ the full subcategory of $\rep_K\A$ which consists of all representations of dimension vector $\df$.\\[1ex]
By defining the affine space $R_{\df}K\Q:= \bigoplus_{\alpha\colon i\rightarrow j}\Hom_K(K^{d_i},K^{d_j})$, one realizes that its points $m$ naturally correspond to representations $M\in\rep_K K\Q(\df)$ with $M_i=K^{d_i}$ for $i\in \Q_0$. 
 Via this correspondence, the set of such representations bound by $I$ corresponds to a closed subvariety $R_{\df}\A\subset R_{\df}K\Q$.\\[1ex]
The algebraic group $\GL_{\df}=\prod_{i\in \Q_0}\GL_{d_i}$ acts on $R_{\df}K\Q$ and on $R_{\df}\A$ via base change, furthermore the $\GL_{\df}$-orbits $\Orb_M$ of this action are in bijection to the isomorphism classes of representations $M$ in $\rep_K\A(\df)$.\\[1ex]
Due to Krull, Remak and Schmidt, every finite-dimensional $\A$-representation decomposes into a direct sum of \textit{indecomposables} (which by definition do not decompose further). For certain classes of finite-dimensional algebras, a convenient tool for the classification of these indecomposable representations is the \textit{Auslander-Reiten quiver} $\Gamma_{\A}=\Gamma(\Q,I)$ of $\rep_K(\Q,I)$. Its vertices $[M]$ are given by the isomorphism classes of indecomposable representations of $\rep_K(\Q,I)$; the arrows between two such vertices $[M]$ and $[M']$ are parametrized by a basis of the space of so-called \textit{irreducible maps} $f\colon M\rightarrow M'$. One standard technique to calculate $\Gamma_{\A}$ is the \textit{knitting process} (see, for example, \cite[IV.4]{ASS}).\\[1ex]
 A component $\Pi_{\A}$ of $\Gamma_{\A}$ is called \textit{preprojective} if it does not contain oriented cycles and if every module in $\Pi_{\A}$ lies in the $\tau$-orbit of some projective. Its corresponding \textit{orbit quiver} $\Upsilon_{\A}$ is defined as the quiver whose vertices are given by the $\tau$-orbits $[X]$ of $\Pi_{\A}$. The number of arrows $[X] \rightarrow [X']$ in $\Upsilon_{\A}$ coincides with the maximal number of arrows $M\rightarrow M'$ in $\Pi_{\A}$, where $M\in [X], M'\in [X']$. The orbit type of $\A$ is defined to be the type of $\Upsilon_{\A}$. A representation $M$ is called \textit{sincere}, if every simple indecomposable of $\A$ comes up in a composition series of $M$.\\[1ex] 
We say that an algebra $\B = K\Q'/I'$ is a \textit{convex subcategory} of $\A = K\Q/I$, if $\Q'$ is a \textit{convex subquiver} of $\Q$ (that is, if two vertices $i,j$ are contained in $\Q'$, then every path of $\Q$ from $i$ to $j$ is completely contained in $\Q'$) and $I'\coloneqq \langle I(i,j)\mid i,j\in \Q'\rangle$. \\[1ex]
An indecomposable projective $P$ has a so-called \textit{separated} radical, if for arbitrary two non-isomorphic direct summands of its radical, their supports are contained in different components of the subquiver $\Q$ obtained by deleting all starting points of paths ending in $i$. We say that $\A$  \textit{fulfills the separation condition}, if every projective indecomposable has a separated radical. If this condition is fulfilled, $\A$ admits a preprojective component, see \cite{Bo4}. In general, the definition of an algebra to be \textit{strongly simply connected} algebra is quite involved. In case of a triangular algebra $\A$, there is an equivalent description, though : $\A$ is \textit{strongly simply connected} if and only if every convex subcategory of $\A$ satisfies the separation condition \cite{Sko2}. 
\subsection{Representation types}\label{reptypetheory}
Consider a finite-dimensional basic $K$-algebra $\A:=K\Q/I$. It is called of \textit{finite representation type}, if there are only finitely many isomorphism classes of indecomposable representations. If it is not of finite representation type, the algebra is of \textit{infinite representation type}. These infinite types split up into two disjoint cases; we say that the algebra $\A$ has
 \begin{itemize}
 \item  \textit{tame representation type} (or \textit{is tame}) if for every integer $n$ there is an integer $m_n$ and there are finitely generated $K[x]$-$\A$-bimodules $M_1,\punkte,M_{m_n}$ which are free over $K[x]$, such that for all but finitely many isomorphism classes of indecomposable right $\A$-modules $M$ of dimension $n$, there are elements $i\in\{1,\punkte,m_n\}$ and $\lambda\in K$, such that  $M\cong  K[x]/(x-\lambda)\otimes_{K[x]}M_i$.
 \item \textit{wild representation type} (or \textit{is wild}) if there is a finitely generated $K\langle X,Y\rangle$-$\A$-bimodule $M$ which is free over $K\langle X,Y\rangle$ and sends non-isomorphic finite-dimensional indecomposable $K\langle X,Y\rangle$-modules via the functor $\_\otimes_{K\langle X,Y\rangle}M$ to non-isomorphic indecomposable $\A$-modules.
\end{itemize} 
In 1979, Drozd proved the following dichotomy statement \cite{Dr}.
\begin{theorem}\label{dichotomy}
 Every finite-dimensional algebra is either tame or wild.
\end{theorem}
The notion of a tame algebra $\A$ yields that there are at most one-parameter families of pairwise non-isomorphic indecomposable $\A$-modules; in the wild case there are parameter families of arbitrary many parameters.\\[1ex]
 A finite-dimensional $K$-algebra is called of \textit{finite growth}, if there is a natural number $m$, such that the indecomposable finite-dimensional modules occur in each dimension $d\geq 2$ in a finite number of discrete and at most $m$ one-parameter families.
\subsubsection{Criteria (via quadratic forms)}\label{quadforms}
For a triangular algebra $\A = K\Q/I$, the \textit{Tits form} $q_{\A}:\mathbf{Z}^{\Q_0}\rightarrow \mathbf{Z}$  is the integral quadratic form defined by 
\[q_{\A}(v) = \sum_{i\in\Q_0} v_i^2 - \sum_{\alpha:i\rightarrow j\in\Q_1} v_iv_j + \sum_{i,j\in\Q_0} r(i,j)v_iv_j;\]
here $r(i,j)$ equals the number of elements in $R\cap I(i,j)$ whenever $R$ is a minimal set of generators of $I$, such that $R\subseteq \bigcup_{i,j\in\Q_0} I(i,j)$. The corresponding symmetric bilinear form is denoted $b_{\A}(\_,\_)$ and fulfills the condition $q_{\A}(v+w)=q_{\A}(v)+b_{\A}(v,w) +q_{\A}(w)$.\\[1ex]
Any non-zero vector $v\in \mathbf{N}^{\Q_0}$ is called \textit{positive}. The quadratic form $q_{\A}$ is called \textit{weakly positive}, if $q_{\A}(v) > 0$ for every $v\in\mathbf{N}^{\Q_0}$; and \textit{(weakly) non-negative}, if $q_{\A}(v) \geq 0$ for every $v\in\mathbf{Z}^{\Q_0}$ (or $v\in\mathbf{N}^{\Q_0}$, respectively).\\[1ex]
For a non-negative form $q_{\A}$, the \textit{radical} of $q_{\A}$ is $\rad q_{\A}:=\{u\in\mathbf{Z}^{\Q_0} \mid q_{\A}(u)=0\}$, we call its elements \textit{nullroots}. In a similar manner, we define the \textit{set of isotropic roots} as $\rad^0 q_{\A}:=\{u\in\mathbf{N}^{\Q_0}\mid q_{\A}(u)=0\}$ and the \textit{set of rational isotropic roots} to be $\rad_{\mathbf{Q}}^0 q_{\A}:=\{u\in\mathbf{Q_+}^{\Q_0}\mid q_{\A}(u)=0\}$ (here, $\mathbf{Q_+}$ is the set of non-negative rational numbers). The maximal dimension of a connected halfspace in $\rad_{\mathbf{Q}}^0 q_{\A}$ is the \textit{isotropic corank} $\corank^0 q_{\A}$ of $q_{\A}$.\\[1ex]
The definiteness of the Tits form is closely related to the representation type of $\A$, and there are connections between roots and certain dimension vectors of representations. Many results are, for example, summarized by De la Pe\~na and Skowro\'{n}ski in \cite{DlPS} where all definitions can be found, too. It is well known that $q_{\A}$ is weakly positive if $\A$ is representation-finite and $q_{\A}$ is weakly non-negative if $\A$ is tame. In certain cases, the opposite directions are true, as well. The following criterion for finite representation type is due to Bongartz \cite{Bo4}.
\begin{theorem}\label{criterionFinite}
Let $\A = K\Q/I$ be a triangular algebra, which admits a preprojective component. Then $\A$ is representation-finite if and only if the Tits form $q_{\A}$ is weakly positive. If the equivalent conditions hold true, then the dimension vector function $X\mapsto \dim X$ induces a bijection between the set of isomorphism classes of indecomposable $\A$-modules and the set of positive roots of $q_{\A}$.
\end{theorem}
Assume that $\Gamma_{\A}$ has a preprojective component. The algebra $\A$ is called  \textit{critical} if $q_{\A}$ is not weakly positive, but every proper restriction of $q_{\A}$ is weakly positive. The term "critical" is actually intuitive, since the conditions of Theorem \ref{criterionFinite} are equivalent to $\A$ not having a convex subcategories which is critical; a classification of the critical algebras can be found in \cite{Bo5,HaVo}.\\[1ex]
 The algebra $\A$ is called \textit{hypercritical} if $q_{\A}$ is not weakly non-negative while every proper restriction of $q_{\A}$ is weakly non-negative. The hypercritical algebras have been classified in Unger's list \cite{Un}. For strongly simply connected algebras, they turn out to be the minimal wild algebras as the following classification of tame types due to of Br\"ustle, De la Pe\~na and Skowro{\'n}ski yields \cite{BdlPS}.
 \begin{theorem}\label{criterionTame}
 Let $\A$ be strongly simply connnected. Then  the following are equivalent:
\begin{enumerate}
\item $\A$ is tame;
\item $q_{\A}$ is weakly non-negative;
\item $\A$ does not contain a full convex subcategory which is hypercritical.
\end{enumerate}
\end{theorem} 
 Theorem \ref{criterionTame} and Theorem \ref{dichotomy} yield a sufficient criterion for wildness.
\begin{corollary}\label{criterionWild}
 Let $\A$ be strongly simply connnected. Whenever there exists $v\in \mathbf{N}^{\Q_0}$, such that $q_{\A}(v)\leq -1$, then $\A$ is of wild representation type.
\end{corollary}
By De la Pe\~na \cite{DlP2}, the following description of finite growth is available.
\begin{proposition}\label{criterionFinGrowth}
Let $\A$ be a strongly simply connected algebra. Then $\A$ is of finite growth if and only if $q_{\A}$ is weakly non-negative and $\corank^0 q_{\A}\leq 1$.
\end{proposition}

 \subsection{Tilted algebras}
 Some of our algebras $\A$ do appear as tilted algebras, that is, there is a hereditary algebra $\B$ and a $\B$-tilting module $T$, such that $\End_{\B}(T) =\A$. In general, there is a sufficient criterion for an algebra to be tilted which reads as follows \cite{Ri}:
\begin{lemma}\label{criterionTiltedSincere}
If $\A$ has a preprojective component which contains an indecomposable sincere representation, then $\A$ is a tilted algebra.
\end{lemma} 
 The following lemma provides a classification for certain cases.
 \begin{lemma}\label{criterionTameTilted}
If there is a preprojective component of $\A$ containing all indecomposable projective $\A$-modules, but no injective module, then the following are equivalent: 
\begin{enumerate}
\item $\A$ is tilted of Euclidean type;
\item $\A$ is tame;
\item $q_{\A}$ is non-negative.
\end{enumerate}
\end{lemma}
Especially the notion of a tame concealed algebra comes up in our setup. These are the Euclidean tilted algebras, such that the tilting module is preprojective; also shown to be the minimal representation-infinite, that is, critical algebras \cite{Ri}. The so-called Bongartz-Happel-Vossieck list (BHV-list) lists the tame concealed algebras \cite{Bo5,HaVo}.

 \section{Staircase algebras}\label{SectStaircase}
Let $n\in\mathbf{N}$ and let $\lambda\vdash n$ be an increasing partition of $n$.  Let $Y(\lambda)$ be the Young diagram of $\lambda$, that is, the box-diagram of which the $i$-th row contains $\lambda_i$ boxes. We denote by $\lambda^{T}$ the transposed increasing partition given by the columns of the Young diagram (from right to left).  Starting with $(1,1)$ in the bottom left corner, we number the boxes of $\lambda$ by $(i,j)$ increasing $i$ from bottom to top and $j$ from left to right.\\[1ex]
Corresponding to $Y(\lambda)$, let us define the quiver $\Q(\lambda)$: its vertices are given by the tuples $(i,j)$ appearing in $Y(\lambda)$; the arrows are given by all horizontal arrows $\alpha_{i,j}: (i,j) \rightarrow (i-1,j)$ and all vertical arrows $\beta_{i,j}: (i,j) \rightarrow (i,j-1)$, whenever all these vertices are defined.
The quiver $\Q(\lambda)$ can be easily visualized by turning the Young diagram $90^{o}$ anti-clockwise and drawing arrows from left to right and from top to bottom.\\[1ex]
Let us define the path algebra $\A(\lambda)=K\Q(\lambda)/I$ with relations given by $I:=I(\lambda)$, which is the $2$-sided admissible ideal generated by all commutativity relations in the appearing squares in $\Q(\lambda)$ which are, if defined, of the form $\beta_{i,j}\alpha_{i,j+1} - \alpha_{i-1,j+1}\beta_{i,j+1}$. Since $I$ is admissible and $\Q(\lambda)$ is connected, $\A(\lambda)$ is a basic, connected,  finite-dimensional $K$-algebra.\\[1ex]
We call $n$ the size of $\A(\lambda)$ and $l:=l(\lambda)$ the length of $\A(\lambda)$, if $\lambda=(\lambda_1,...,\lambda_l)$. For easier notation, we merge similar entries of $\lambda$ by potencies, for example $(1,1,2,2,2,6,8,8)=: (1^2,2^3,6,8^2)$. Then the length is given by the number of entries in original notation; the number of entries in potency-notation is called the number of steps $s:=s(\lambda)$ of $\A(\lambda)$. 
Since the quivers $\Q(\lambda)$ look like staircases, the following definition is reasonable.
\begin{definition}\label{defstaircase}
Let $\A$ be a finite-dimensional $K$-algebra. It is called
\begin{itemize}
\item $n$-staircase algebra, if there is an increasing partition $\lambda\vdash n$, such that $\A\cong \A(\lambda)$.
\item staircase algebra, if there exists a natural number $n$, such that $\A$ is an $n$-staircase algebra.
\item $m$-step, if the number of steps equals $m$.
\end{itemize}
\end{definition}
We denote the Tits quadratic form by $q_{\lambda}:=q_{\A(\lambda)}$ and the Auslander-Reiten quiver by $\Gamma(\lambda):=\Gamma_{\A(\lambda)}$.\\[1ex]
Let us consider an example to illustrate these definitions before discussing staircase algebras and their properties in detail.
\begin{example}\label{exStaircase}
 Let $n=7$ and $\lambda=(1,1,2,3)$ and its Young diagram drawn in Example \ref{gradedExample}. Its quiver $\Q(\lambda)$ is given by
\begin{center}\small\begin{tikzpicture}
\matrix (m) [matrix of math nodes, row sep=1.81em,
column sep=1.8em, text height=0.9ex, text depth=0.1ex]
{  &  &  &  (1,3) & &  & & &  \bullet \\
 &   & (2,2) & (1,2)&\hat{=} &  & & \bullet & \bullet \\
(4,1) & (3,1) & (2,1) &  (1,1) & &\bullet &\bullet & \bullet & \bullet\\
 };
\path[->]

(m-2-3) edge node[above]{$\alpha_{2,2}$} (m-2-4)
(m-2-3) edge node[right]{$\beta_{2,2}$} (m-3-3)
(m-1-4) edge node[right]{$\beta_{1,3}$} (m-2-4)
(m-2-4) edge node[right]{$\beta_{1,2}$} (m-3-4)
(m-3-1) edge node[above]{$\alpha_{4,1}$}  (m-3-2)
(m-3-2) edge node[above]{$\alpha_{3,1}$}  (m-3-3)
(m-3-3) edge node[above]{$\alpha_{2,1}$} (m-3-4)
(m-2-8) edge (m-2-9)
(m-2-8) edge (m-3-8)
(m-1-9) edge (m-2-9)
(m-2-9) edge (m-3-9)
(m-3-6) edge  (m-3-7)
(m-3-7) edge (m-3-8)
(m-3-8) edge (m-3-9);
\path[-]
(m-2-8) edge[dotted]  (m-3-9)
;\end{tikzpicture}\end{center} 
The $7$-staircase algebra $\A(\lambda)$ of $3$ steps and of length $4$ is defined by
\[\A(\lambda)= K\Q(\lambda)/(\beta_{1,2}\alpha_{2,2} - \alpha_{2,1}\beta_{2,2}) .\]
\end{example}

\subsection{Link to graded nilpotent pairs}\label{link}
Each graded nilpotent pair $(\varphi, \psi)$ together with a bigraded vector space $V$ as in \ref{gradNilpTheory} can be considered as a representation $M:=M(\varphi, \psi,V)$ of $\A(\lambda(V))$ in a natural way. We denote $\underline{\dim}V:= \underline{\dim}_{\A(\lambda(V))}M$ which depends on the bigrading on $V$, but not on $(\varphi, \psi)$.

\begin{example}
Consider the setup of Example \ref{gradedExample}. Then  the $\A((1^2,2,4))$-representation $M(\varphi, \psi, V)$  is given by
\begin{center}\small\begin{tikzpicture}
\matrix (m) [matrix of math nodes, row sep=1.81em,
column sep=1.8em, text height=0.9ex, text depth=0.1ex]
{ &  &  &  K \\
  &  & K^3 & K^2 \\
 K & K & K^2 & 0 \\
 };
\path[->]
(m-1-4) edge node[right]{$\psi_{1,3}$} (m-2-4)
(m-2-3) edge node[above]{$\varphi_{2,2}$}  (m-2-4)
(m-2-3) edge node[left]{$\psi_{2,2}$} (m-3-3)
(m-2-4) edge  (m-3-4)
(m-3-1) edge node[below]{$\varphi_{4,1}$} (m-3-2)
(m-3-2) edge  node[below]{$\varphi_{3,1}$} (m-3-3)
(m-3-3) edge  (m-3-4)
(m-2-3) edge[-,dotted]  (m-3-4)
;\end{tikzpicture}\end{center} 
\end{example}

 Representation-theoretically speaking, the graded nilpotent pairs of $V$ are encoded in the representation variety $R_{\underline{\dim}V}\A(\lambda(V))$. Therefore, certain criteria can be translated from the representation theory of finite-dimensional algebras straight away.
\begin{theorem}\label{gradedCriterionLambda}
Let $\lambda$ be a partition. 
Modulo Levi-base change, there are only finitely many $\lambda$-graded nilpotent pairs  if and only if  $\A(\lambda)$ is of finite representation type.
Otherwise, there are at most one-parameter families of non-decomposable graded nilpotent pairs if and only if $\A(\lambda)$ is tame.
\end{theorem}
\begin{lemma}\label{gradedFiniteDim}
The number of graded nilpotent pairs (up to Levi-base change) of the fixed bigraded vector space $V$ is finite if and only if $R_{\underline{\dim}V}\A(\lambda(V))$ admits only finitely many $\GL_{\underline{\dim}V}$-orbits.
\end{lemma}
\begin{corollary}
The equivalent conditions of Lemma \ref{gradedFiniteDim} hold true if and only if the number of isomorphism classes of representations in $\rep_K\A(\lambda(V))(\underline{\dim} V)$ is finite.
\end{corollary}
These criteria motivate the classification of representation types of staircase algebras in Section \ref{SectRepTypes} and certain further considerations in Section \ref{SectCasestudy}. To start, we study general properties of staircase algebras.
\subsection{General properties}\label{properties}
For $(i,j)\in \Q(\lambda)_0$, let $S(i,j)$ be the standard simple representation at the vertex $(i,j)$ of $\A(\lambda)$. The (isomorphism classes of the) projective indecomposables of $\A(\lambda)$ are parametrized by $P(i,j)$, $(i,j)\in \Q(\lambda)_0$, which are given by
 \[P(i,j)_{k,l}=\left\lbrace
\begin{array}{ll}
K & {\rm if}~ k\leq i ~{\rm and}~ l\leq j, \\ 
0 & {\rm otherwise}.
                                                               \end{array}
 \right. \]
together with identity and zero maps, accordingly.
\begin{proposition}\label{basicproperties}
The algebra $\A(\lambda)$ is triangular, fulfills the separation condition and is strongly simply connected.
\end{proposition}
\begin{proof}
Since $\Q(\lambda)$ does not contain oriented cycles, $\A(\lambda)$ is triangular.\\ The radical of every projective indecomposable is indecomposable, such that each projective indecomposable has separated radical and $\A(\lambda)$ fulfills the separation condition.\\ 
Let $\lambda$ be an increasing partition and consider a convex subcategory $C$ of $A(\lambda)$. We aim to show that $C$ fulfills the separation condition. 

The radical of every projective $P$ of the vertex $(i,j)$ is either indecomposable or decomposes into exactly two indecomposables. The latter is the case if and only if $(i-1,j-1)$ is not a vertex of the quiver $\Q$ of $C$. Let $\Q'$ be the subquiver of $\Q$ which is obtained by deleting all start vertices of all paths ending in $(i,j)$. It is clear that $\Q'$ decomposes into two disjoint quivers, these correspond to the supports of the two indecomposable direct summands of $\rad P$.  Thus, $C$ fulfills the separation condition. The algebra  $\A(\lambda)$ is, thus, strongly simply connected.
\end{proof}
We approach the module categories by general techniques. Since Proposition \ref{basicproperties} states that $\A(\lambda)$ fulfills the separation condition, the following lemma follows from Section \ref{reptheory}).
\begin{lemma}\label{existencepreproj}
The algebra $\A(\lambda)$ admits a preprojective component $\Pi(\lambda):= \Pi_{\A(\lambda)}$.
\end{lemma}
The knowledge of the orbit type, that is, of the  type of the orbit quiver $\Upsilon(\lambda):=\Upsilon_{\A(\lambda)}$ gives first clues about the corresponding representation types, which we classify in Theorem \ref{reptype}.
\begin{lemma}\label{orbittype}
The orbit type of $\A(\lambda)$ is 
\begin{enumerate}
\item $A_n$, if $\lambda=(1^k,n-k)$ for some $k$,
\item $D_n$, if $\lambda\in\{(1^{n-4},2^2), (2,n-2)\}$,
\item $E_6$, if $\lambda\in\{(1,2,3), (2^3), (3^2)\}$,
\item $E_7$, if $\lambda\in\{(1^2,2,3), (1,2,4), (1,2^3), (3,4), (1,3^2), (2^2,3)\}$,
\item $E_8$, if $\lambda\in\{(1^3,2,3), (1,2,5), (1^2,2^3), (3,5), (1^2,3^2), (2^2,4), (2^4), (4^2)\}$,
\item $\widetilde{E_7}$, if $\lambda\in\{(1,3,4), (1,2^2,3), (2,3^2), (1^2,2,4)\}$; and if $\lambda\in\{(3^3), (1,4^2), (2^3,3)\}$.
\item $\widetilde{E_8}$, if  $\lambda\in\{(3,6), (4,5), (1,2,6), (2^2,5), (1^3,3^2), (1,2^4), (1^4,2,3), (1^3,2^3)\}$ ; and if $\lambda=(5^2)$.
\end{enumerate}
In every other case, the orbit type is wild. Note that for every listed case except for $\lambda\in\{(3^3), (1,4^2), (2^3,3), (5^2)\}$, every indecomposable projective indecomposable comes up in the preprojective component. In the Euclidean cases, no injective indecomposable is contained in the preprojective component.
\end{lemma}
\begin{proof}
The Auslander-Reiten quivers of every staircase algebra appearing in 1. to 5., that is, of every staircase algebra of Dynkin orbit type, are depicted in the Appendix \ref{appendix}. The orbit type can be read off directly.\\[1ex]
Let $n\in\{8,9\}$, then the orbit quivers listed in 6. and 7. can easily be calculated by knitting the beginning of the preprojective components, since every projective indecomposable comes up. Every remaining case is seen to be of of wild orbit type by the same method. There are two exceptions of this procedure:\\
 In case $\lambda=(3^3)$ the claim follows  from knitting - but one has to realize that the projective indecomposable $P(3,3)$ is injective as well and does not appear in the preprojective component, since this case is representation-infinite, see Lemma \ref{reptypel3}.\\
  In case $\lambda = (1,4^2) = (2^3,3)^T$ knitting yields the claim, but it must be proved that $P(2,4)$ is not preprojective. One method to show this fact is to knit the Auslander-Reiten quiver $\Gamma_{3,1}$ restricted to the vertex $(3,1)$; it is cyclic and there are indecomposables $U_1,..,U_8$, such that every $\tau^{-}$-translation of these will be non-zero-dimensional at the vertex $(3,1)$. Thus, it suffices to calculate the Auslander-Reiten quiver until $U_1,..,U_8$ appear and realizes that the radical of $P(2,4)$ does not appear.\\[1ex]
For $n\geq 10$, every case except $(5^2)$ is of wild orbit type. The case $(5^2)$ is of type $\widetilde{E_8}$, which can be seen by knitting its preprojective component. Again, one of the projective indecomposables is injective and  does not appear in the preprojective component - since this case is representation-infinite by Lemma \ref{reptypel3}.
\end{proof}

\section{Representation types}\label{SectRepTypes}
\subsection{Reductions}\label{reductions}
In order to classify the representation types of all staircase algebra, it is of great value to have certain reduction statements available.
\begin{lemma}\label{reductioncomp}
 Let $\A$ be a convex subcategory of $\A'$. Then
 \begin{enumerate}
\item ... if $\A$ is tame, then $\A'$ is tame or wild.
 \item ... if $\A$ is wild, then $\A'$ is wild.
\item ... if $\A'$ has finite representation type, $\A$ has finite representation type.
\item ... if $\A'$ is tame, then $\A$ is tame or of finite representation type.
\end{enumerate}
 
 In particular, if $\lambda \leq \lambda'$ is a subpartition, then these facts hold true for $\A=\A(\lambda)$ and $\A'=\A(\lambda')$.

\end{lemma}
\begin{proof}
 The claim follows from general representation theory of quivers with relations \cite{ASS} by restricting in a natural way or expanding with zero. 
\end{proof}
\begin{lemma}\label{reductionsymm}
For each partition $\lambda$, the representation type of $\A(\lambda)$ and $\A(\lambda^{T})$ is the same. 
\end{lemma}

\subsection{Classification}
We provide a complete classification of the representation type of a staircase algebra. In order to assure a nice structure of the presentation, we begin by classifying the representation types of staircase algebras of length $2$ in Lemma \ref{reptypel2} and of length $3$ in Lemma \ref{reptypel3}, since most of the staircase algebras of finite and tame representation type come up in these contexts. We then generalize the results to arbitrary staircase algebras in Theorem \ref{reptype} which completes the classification. 
\begin{lemma}\label{reptypel2}
 A  staircase algebra $\A(\lambda)$ of length $2$, that is, $\lambda=(\lambda_1,\lambda_2)\vdash n$, is
 \begin{itemize}
 \item of finite representation type if and only if $n\leq 8$ or $\lambda_1\in\{1,2\}$.
 \item tame concealed if and only if $\lambda=(3,6)$.
 \item tame, but not tame concealed, if and only if $\lambda\in\{(4,5), (5^2)\}$.
\end{itemize} 
Otherwise, $\A(\lambda)$ is of wild representation type.
\end{lemma}

\begin{proof}
The following table shows the structure of our proof; the marked ones show which cases need to be proved in order to classify every remaining non-marked case due to reductions via Lemma \ref{reductioncomp} and Lemma \ref{reductionsymm} (W=Wild, T=Tame, TC= Tame concealed, F=Finite):

\begin{center}
\begin{tabular}{l||l|l|l|l|l|l|l|l|l|l|}
\hline
\vdots & \cellcolor{grey}\vdots & \cellcolor{grey}\vdots & \cellcolor{grey}\vdots & \cellcolor{grey}\vdots & \cellcolor{grey}\vdots & \cellcolor{lightviolet}\vdots &\cellcolor{lightviolet}\vdots&\cellcolor{lightviolet}\vdots&\cellcolor{lightviolet}\vdots\\ \hline
6 & \cellcolor{grey}x & \cellcolor{grey}x & \cellcolor{grey}x & \cellcolor{grey}x & \cellcolor{grey}x & \cellcolor{lightviolet}W &\cellcolor{lightviolet}W &\cellcolor{lightviolet}W & \cellcolor{lightviolet}$\cdots$\\ \hline
5 & \cellcolor{grey}x & \cellcolor{grey}x & \cellcolor{grey}x & \cellcolor{grey}x & \cellcolor{dlightblue}\textbf{T} & \cellcolor{lightviolet}W & \cellcolor{lightviolet}W & \cellcolor{lightviolet}W & \cellcolor{lightviolet}$\cdots$\\ \hline
4 & \cellcolor{grey}x & \cellcolor{grey}x & \cellcolor{grey}x & \cellcolor{dlightgreen}\textbf{F} & \cellcolor{dlightblue}\textbf{T} & \cellcolor{dlightviolet}\textbf{W} & \cellcolor{lightviolet}W & \cellcolor{lightviolet}W & \cellcolor{lightviolet}$\cdots$\\ \hline
3 & \cellcolor{grey}x & \cellcolor{grey}x & \cellcolor{lightgreen}F & \cellcolor{lightgreen}F & \cellcolor{dlightgreen}\textbf{F} & \cellcolor{dlightblue} \textbf{TC} & \cellcolor{dlightviolet}\textbf{W} & \cellcolor{lightviolet}W &\cellcolor{lightviolet}  \cellcolor{lightviolet}$\cdots$\\ \hline
2 & \cellcolor{grey}x & \cellcolor{lightgreen}F &\cellcolor{lightgreen}F & \cellcolor{lightgreen}F & \cellcolor{lightgreen}F &\cellcolor{dlightgreen}\textbf{F} & \cellcolor{dlightgreen}\textbf{F}& \cellcolor{dlightgreen}\textbf{F}& \cellcolor{dlightgreen}$\cdots$\\ \hline
1 & \cellcolor{lightgreen}F & \cellcolor{lightgreen}F & \cellcolor{lightgreen}F & \cellcolor{lightgreen}F &\cellcolor{lightgreen}F & \cellcolor{lightgreen}F &\cellcolor{lightgreen}F& \cellcolor{lightgreen}F& \cellcolor{dlightgreen}$\cdots$\\ \hline\hline 
$\lambda_1/\lambda_2$ & 1 & 2 & 3 & 4 & 5 & 6 & 7 & 8 & $\cdots$
\end{tabular}
\end{center}
Representation-finite cases: Let $\lambda_1=1$, then $\Q(\lambda)$ is of type $A_n$ which is representation-finite. 
 For the remaining maximal finite cases, the Auslander-Reiten quivers (or their symmetric versions), are attached in the Appendix \ref{appendix}. In particular, if $\lambda_1=2$, then the Auslander-Reiten quiver $\Gamma(\lambda)$ for $\lambda=(2,6)$ is attached and it is easy to see that the quiver stays finite, if $\lambda_2$ increases.\\[1ex]
Representation-infinite cases: The case $(3,6)$ is tame concealed by \cite{HaVo}. \\[1ex]
The cases $\lambda\in\{(5^2),(4,5)\}$ are tame, since the algebras $\A(\lambda)$ do not contain a hypercritical convex subcategory (Unger's list \cite{Un}) by Theorem \ref{criterionTame} (see Lemma \ref{MaxTameStraightForw3} for a straight-away proof of their tameness). They are not minimal tame and thus not tame concealed \cite{Ri}, since they contain a tame concealed subquiver of the BHV-list \cite{BGRS}:
\begin{center}\small\begin{tikzpicture}[descr/.style={fill=white}]
\matrix (m) [matrix of math nodes, row sep=0.8em,
column sep=1.0em, text height=0.8ex, text depth=0.1ex]
{ &\bullet & \bullet & \bullet & \bullet\\
\bullet & \bullet & \bullet & \bullet  &   \\};
\path[-]
(m-1-2) edge  (m-1-3)
(m-1-3) edge (m-1-4)
(m-1-4) edge (m-1-5)
(m-2-1) edge  (m-2-2)
(m-2-2) edge  (m-2-3)
(m-2-3) edge (m-2-4)
(m-1-2) edge (m-2-2)
(m-1-3) edge (m-2-3)
(m-1-4) edge (m-2-4)
(m-1-2) edge[-,dotted] (m-2-3)
(m-1-3) edge[-,dotted] (m-2-4)
;\end{tikzpicture}\end{center} 
The cases $\lambda=(4,6)$ and $\lambda'=(3,7)$ are wild by Corollary \ref{criterionWild}, since
\[q_{\lambda}\left(\begin{array}{cccccc}
&&2& 4 &4 &2\\ 
 1& 2 &4 &4& 2& 0
\end{array}\right) =  q_{\lambda'}\left(\begin{array}{cccccccc}
&&&&4& 6 &4 \\ 
1& 2 &4 &6& 8& 6 & 2
\end{array}\right) = -1  \]
and explicit $2$-parameter families are presented in Remark \ref{explicit2Param}.
\end{proof}

\begin{lemma}\label{reptypel3}
 A staircase algebra $\A(\lambda)$ of length $3$, that is,  $\lambda=(\lambda_1,\lambda_2,\lambda_3)\vdash n$, is
\begin{itemize}
\item of finite representation type if and only if 
$n\leq 7$ or $\lambda\in\{(1,1,\lambda_3), (1,2,5), (2^2,4)\}$.
\item tame concealed if and only if $\lambda\in\{(1,3,4), (1,2,6), (2^2,5), (1^3,3^2)\}$.
\item tame, but not tame concealed, if and only if 
$\lambda\in\{(1,4^2), (2,3^2), (2^3,3), (3^3)\}$.
\end{itemize}
Otherwise, $\A(\lambda)$ is of wild representation type.
 \end{lemma}

\begin{proof}
Whenever $\lambda_1\geq 3$, either $\lambda=(3^3)$, or $\A(\lambda)$ is of wild representation type: the case $\lambda=(2,3,4)$ is wild by Corollary \ref{criterionWild}
\[q_{\lambda}\left(\begin{array}{cccc}
&  &2 &2\\ 
& 2 &4 &2\\ 
1 &3& 3& 1
\end{array}\right) = -1.  \]

If $\lambda=(3^3)$, then $\A(\lambda)$  is not wild, since it does not contain a hypercritical convex subcategory (Unger's list \cite{Un}). See Lemma \ref{MaxTameStraightForw3} for a straight forward proof of its tameness. It is  not tame concealed, since it contains the tame concealed algebra
\begin{center}\small\begin{tikzpicture}[descr/.style={fill=white}]
\matrix (m) [matrix of math nodes, row sep=0.8em,
column sep=1.0em, text height=0.8ex, text depth=0.1ex]
{   & \bullet & \bullet\\
 \bullet & \bullet & \bullet\\
\bullet  & \bullet & \\};
\path[->]
(m-1-2) edge[=]  (m-1-3)
(m-2-1) edge  (m-2-2)
(m-2-1) edge  (m-3-1)
(m-2-2) edge  (m-2-3)
(m-2-2) edge  (m-3-2)
(m-3-1) edge  (m-3-2)
(m-1-2) edge (m-2-2)
(m-1-3) edge (m-2-3)
(m-1-2) edge[-,dotted] (m-2-3)
(m-2-1) edge[-,dotted] (m-3-2)
;\end{tikzpicture}\end{center}  
By reduction via Lemma \ref{reductioncomp} we only consider $\lambda_1\in\{1,2\}$. The following tables visualize all cases and it suffices to prove the marked ones by  Lemma \ref{reductioncomp}.
\begin{center}
\begin{tabular}{|l||l|l|l|l|l|l|l|l|l|}
\hline
$\lambda_1=1$ & \cellcolor{grey}\vdots & \cellcolor{grey}\vdots & \cellcolor{grey}\vdots & \cellcolor{grey}\vdots & \cellcolor{grey}\vdots & \cellcolor{lightviolet}\vdots &\cellcolor{lightviolet}\vdots &\cellcolor{lightviolet}\vdots &\cellcolor{lightviolet}\vdots\\ \hline
6 & \cellcolor{grey}x & \cellcolor{grey}x & \cellcolor{grey}x & \cellcolor{grey}x & \cellcolor{grey}x & \cellcolor{lightviolet}W &\cellcolor{lightviolet}W &\cellcolor{lightviolet}W & \cellcolor{lightviolet}$\cdots$\\ \hline
5 & \cellcolor{grey}x & \cellcolor{grey}x & \cellcolor{grey}x & \cellcolor{grey}x & \cellcolor{lightviolet}W & \cellcolor{lightviolet}W & \cellcolor{lightviolet}W & \cellcolor{lightviolet}W & \cellcolor{lightviolet}$\cdots$\\ \hline
4 & \cellcolor{grey}x & \cellcolor{grey}x & \cellcolor{grey}x & \cellcolor{dlightblue}\textbf{T} & \cellcolor{lightviolet}W & \cellcolor{lightviolet}W & \cellcolor{lightviolet}W & \cellcolor{lightviolet}W & \cellcolor{lightviolet}$\cdots$\\ \hline
3 & \cellcolor{grey}x & \cellcolor{grey}x & \cellcolor{dlightgreen}\textbf{F} & \cellcolor{dlightblue}\textbf{TC} & \cellcolor{dlightviolet}\textbf{W} & \cellcolor{lightviolet}W & \cellcolor{lightviolet}W & \cellcolor{lightviolet}W & \cellcolor{lightviolet}$\cdots$\\ \hline
2 & \cellcolor{grey}x & \cellcolor{lightgreen}F &\cellcolor{lightgreen}F & \cellcolor{lightgreen}F & \cellcolor{dlightgreen}\textbf{F} &\cellcolor{dlightblue}\textbf{TC} & \cellcolor{dlightviolet}\textbf{W}& \cellcolor{lightviolet}W& \cellcolor{lightviolet}$\cdots$\\ \hline
1 & \cellcolor{lightgreen}F & \cellcolor{lightgreen}F & \cellcolor{lightgreen}F & \cellcolor{lightgreen}F &\cellcolor{lightgreen}F & \cellcolor{lightgreen}F &\cellcolor{lightgreen}F& \cellcolor{dlightgreen}\textbf{F}& \cellcolor{dlightgreen}$\cdots$\\ 
\hline
\hline
$\lambda_1=2$ &\cellcolor{grey}\vdots&  \cellcolor{grey}\vdots & \cellcolor{grey}\vdots & \cellcolor{grey}\vdots & \cellcolor{grey}\vdots & \cellcolor{lightviolet}\vdots &\cellcolor{lightviolet}\vdots&\cellcolor{lightviolet}\vdots&\cellcolor{lightviolet}\vdots\\ \hline
4 &\cellcolor{grey}x & \cellcolor{grey}x & \cellcolor{grey}x & \cellcolor{lightviolet}W & \cellcolor{lightviolet}W & \cellcolor{lightviolet}W & \cellcolor{lightviolet}W & \cellcolor{lightviolet}W & \cellcolor{lightviolet}$\cdots$\\ \hline
3 &\cellcolor{grey}x &\cellcolor{grey}x & \cellcolor{dlightblue}\textbf{T} & \cellcolor{dlightviolet}\textbf{W} & \cellcolor{lightviolet}W & \cellcolor{lightviolet}W & \cellcolor{lightviolet}W & \cellcolor{lightviolet}W & \cellcolor{lightviolet}$\cdots$\\ \hline
2 &\cellcolor{grey}x  &\cellcolor{lightgreen}F &\cellcolor{lightgreen}F & \cellcolor{dlightgreen}\textbf{F} & \cellcolor{dlightblue}\textbf{TC} &\cellcolor{dlightviolet}\textbf{W} & \cellcolor{lightviolet}W& \cellcolor{lightviolet}W& \cellcolor{lightviolet}$\cdots$\\ \hline
$\lambda_2/\lambda_3$ & 1&  2 & 3 & 4 & 5 & 6 & 7 & 8 & $\cdots$\\ \hline
\end{tabular}
\end{center}

Representation-finite cases: Either $\lambda_2=1$, then we arrive at an $A_n$-classification. For the remaining finite cases, the Auslander-Reiten quivers or their symmetric versions (by Lemma \ref{reductionsymm})  are depicted in the Appendix \ref{appendix}.\\[1ex]
Representation-infinite cases: \\[1ex]
Every marked tame concealed case is tame concealed by the BHV-list \cite{HaVo}. \\[1ex]
The case $\lambda=(1,4^2)$ is not wild by Theorem \ref{criterionTame}, since the algebra $\A(\lambda)$ does not contain any hypercritical convex subcategory (Unger's list \cite{Un}). It is representation-infinite and not tame concealed, that is, minimal tame, since it contains the case $(1,3,4)$ which is tame.\\[1ex]
The case $\lambda=(2,3^2)$ is tame, but not tame concealed since $\A(3^3)$ is tame and since it contains the above depicted convex subcategory.\\[1ex]
For every minimal wild case, by Corollary \ref{criterionWild} we provide dimension vectors $v$ which fulfill $q_{\lambda}(v)\leq -1$:
\begin{center}

\begin{tabular}{|c|c|c|}
\hline
$\small\begin{array}{cccc}
 & &1\\ 
& &2\\ 
 & 2& 4\\
 & 4& 4\\
2 & 4& 2\\
\end{array}$& $\small\begin{array}{cccc} 
& &1\\ 
& &2\\ 
& &4\\ 
& &6\\ 
 &&8\\
 & 6&10\\
4 & 8& 6\\
\end{array}$& $\small\begin{array}{ccc}
& &1\\ 
& &2\\
& &4\\
 & & 6\\
2 & 6& 8\\
4 & 6& 4\\
\end{array}$\\\hline
$(1,3,5)$& $(1,2,7)$& $(2^2,6)$\\ \hline

\end{tabular}

\end{center}
The case $\lambda=(2,3,4)$ has been shown to be wild above.
\end{proof}
We are able to deduce all remaining cases now which leads to a complete classification of the representation types of all staircase algebras.
\begin{theorem}\label{reptype}
A staircase algebra $\A(\lambda)$ is 
\begin{itemize}
\item representation-finite if and only if one of the following holds true:
\begin{enumerate}
 \item\label{1stcase} $\lambda\in\{(n),~ (1^k,n-k),~ (2,n-2),~ (1^{n-4},2^2)\}$ for some $k\leq n$,
\item\label{2ndcase}  $n\leq 8$ and $\lambda\notin\{(1,3,4),~ (2,3^2),~ (1,2^2,3),~ (1^2,2,4)\}$.
\end{enumerate}
\item tame concealed if and only if $\lambda$ comes up in the following list:\\ $(3,6)$, $(1,2,6)$, $(1,3,4)$, $(2^2,5)$, $(1^2,2,4)$, $(1,2^2,3)$, $(1^3,3^2)$, $(1^3,2^3)$, $(1^4,2,3)$.
\item tame, but not tame concealed if and only if $\lambda$ comes up in the following list:\\ $(4,5)$, $(5^2)$, $(1,4^2)$, $(2,3^2)$, $(3^3)$, $(2^3,3)$, $(1,2^4)$, $(2^5)$.
\end{itemize} 
 Otherwise, $\A(\lambda)$ is of wild representation type. 
\end{theorem}
\begin{proof}
All listed representation-finite cases are in fact of finite representation type:
Let $\lambda=(1^k,n-k)$ for some $k$, then we obtain an $A_n$-classification problem, thus, representation-finiteness \cite{Ga1}. Given the two symmetric cases $\lambda=(2,n-2)=(1^{n-4},2^2)^T$, finiteness follows from  Lemma \ref{reptypel2}. For every remaining finite case, the Auslander-Reiten quivers can be found in the Appendix \ref{appendix}.\\[1ex]
Assume, $l(\lambda)\leq 3$, then the claimed classification follows from Lemma \ref{reptypel2} and Lemma \ref{reptypel3}.  Thus, let without loss of generality $\lambda=(\lambda_1,..., \lambda_k)$, such that $k\geq 4$; and $\lambda_{k-1}>1$ (otherwise we arrive at a finite case).\\[1ex]
If $\lambda_{k-1}\geq 4$, then $\A(\lambda)$ is of wild representation type by Lemma \ref{reductioncomp}: The case  $\A(1,1,3,4)$ is wild by Corollary \ref{criterionWild}, since

	\[q_{\lambda}\left(\begin{array}{cccc}
	&  & &2\\
&  &2 &4\\ 
&  &4 &4\\ 
1 &2& 4& 2
\end{array}\right) = -1.  \]
If $\lambda_{k-1}=2$, then $\lambda$ is of the form $(1,...,1,2,...,2,x)$ for some $x$.
\begin{itemize}
\item If $x\in\{2,3\}$, then the transpose of every case has been considered in Lemma \ref{reptypel2} or Lemma \ref{reptypel3}.
\item If $x\geq 4$, then 
\begin{itemize}
\item for $\lambda_{k-2}=2$  the algebra is wild due to reduction by Lemma \ref{reductioncomp} since $\A((1,1,3,4)^T)=\A(1,2,2,4)$ is wild via symmetry of Lemma \ref{reductionsymm}. 
\item if $\lambda_{k-2}=1$ and $x=4$, the algebra $\A(\lambda)$ is tame concealed \cite{HaVo}.
\item for $\lambda_{k-2}=1$ and $x\geq 5$  the algebra is wild by  Corollary \ref{criterionWild}, since
	\[q_{\lambda}\left(\begin{array}{cccc}
	&  & &1\\
	&  & &2\\ 
&  & &4\\ 
&  &4 &6\\ 
2 &4& 6& 4
\end{array}\right) = -1. \]
\end{itemize} 
\end{itemize}
If $\lambda_{k-1}=3$, then (since $\A(1,1,3,4)$ is wild) $\A(\lambda)$ is wild whenever $\lambda_k\geq 4$. If $\lambda_k=3$, the transpose of $\lambda$ is of length $3$ and has, thus, been considered in Lemma \ref{reptypel3}.
\end{proof}

\begin{corollary}\label{blockstalg}
The tensor algebra $KA_m\otimes_K KA_l$  is of finite representation type if and only if  $(m,l)\in\{(1,l), (m,1), (2,2), (2,3), (3,2), (4,2), (2,4)\}$. It is tame exactly if $(m,l)\in\{(2,5), (5,2), (3,3)\}$ and wild, otherwise.
\end{corollary}
\begin{proof}
The tensor algebra $KA_m\otimes_K KA_l$ equals the $1$-step staircase algebra $\A(k^l)$, thus, the proof follows from Theorem \ref{reptype}.
\end{proof}
The hierarchy of cases is depicted in the Appendix \ref{hierarchy}.
Note that the representation type  of the staircase algebra $\A((2^k))$ has been examined - it is already known by Asashiba; and by Escolar and Hiraoka \cite{EH}.\\[1ex]
 By Lemma \ref{orbittype}, Theorem \ref{reptype} yields the following classification via orbit types.
\begin{lemma}\label{classificationType}
$\A(\lambda)$ is of finite / tame / wild representation type if and only if its orbit quiver is Dynkin / Euclidean / wild.
\end{lemma}
These results were obtained in the $\GL_n$-setup. We think that it should be possible to generalize them to arbitrary classical Lie types by methods similar to \cite{HiRoe}.
\section{Properties depending on representation types}\label{SectCasestudy}
We have a closer look at the representation theory of staircase algebras now and divide our considerations by representation types, which are known from Theorem \ref{reptype}.
\subsection{Finite cases}\label{finiteCases}
All finite cases are listed in the table, together with  their orbit quiver $\Upsilon(\lambda)$ (see Lemma \ref{orbittype}) and the link to their Auslander-Reiten quiver $\Gamma(\lambda)$ in the Appendix \ref{appendix}). In the latter, the corresponding maximal dimension vectors are marked. Since $\Gamma(\lambda)$ always contains a sincere representation, each such algebra $\A(\lambda)$ is tilted Dynkin and the orbit quiver indicates the frame. By Theorem \ref{criterionFinite} up to isomorphism, every indecomposable appears in $\Gamma(\lambda)$; such that they can be constructed directly or by means of methods from Auslander-Reiten Theory \cite{ASS}. This way, a complete representative system of indecomposable modules is obtained for each finite case. Clearly, all orbits of a fixed dimension vector can be deduced from these by Krull-Remak-Schmidt.
By results of Bongartz \cite{Bo2}, the explicit calculation of orbit closures can nicely be done.
\begin{center}
\begin{tabular}[h]{|c|c|c|c||c|c|c|c|}
\hline
$n$ & $\lambda$  &     $\Upsilon(\lambda)$ & $\Gamma(\lambda)$ & $n$ & $\lambda$  &     $\Upsilon(\lambda)$ & $\Gamma(\lambda)$\\ 
\hline 

$n$& $\begin{array}{c}
(1^{n-k},k)
\end{array}$ 
 & $A_n$ & -

 \\
\hline

$n$& $\begin{array}{c}
(1^{n-4},2^2) \\
(2,n-2)\end{array}$
 & $D_n$ & \ref{ark26}&$7$&$\begin{array}{c}
(1,2,4) \\
(1^2,2,3)\end{array}$&

  $E_7$ & \ref{ark124}\\ \hline

$8$& $\begin{array}{c}
(1,2,5) \\
(1^3,2,3)\end{array}$
&   $E_8$  & \ref{ark125}&$7$
&$\begin{array}{c}
(3,4) \\
(1,2^3)\end{array}$& 

  $E_7$ & \ref{ark34}
 \\
\hline

$8$ &$\begin{array}{c}
 (3,5)\\
(1^2,2^3)\end{array}$ & $E_8$ & \ref{ark35}&$7$& $\begin{array}{c}
(1,3^2) \\
(2^2,3)\end{array}$
& $E_7$  & \ref{ark133}
 \\
\hline

$8$&$\begin{array}{c}
(1^2,3^2) \\
(2^2,4)\end{array}$
& $E_8$ & \ref{ark224}&$6$& $\begin{array}{c}
 (1,2,3)
\end{array}$

& $E_6$ & \ref{ark123}
 \\
\hline

$8$
 &$\begin{array}{c}
 (2^4),
(4^2)\end{array}$

 & $E_8$ & \ref{ark2222}& $6$
&$\begin{array}{c}
(2^3),
(3^2)\end{array}$

& $E_6$ & \ref{ark222}
 \\
\hline
\end{tabular}
\end{center}
\subsection{Tame concealed algebras}\label{TC}
Given a tame concealed algebra $\A(\lambda)$, we know that there is an algebra $\B$ and a preprojective $\B$-tilting module $T$, such that $\End_{\B}(T) =\A(\lambda)$. 
\begin{center}
\includegraphics[trim=130 485 185 125,clip,width=290pt]{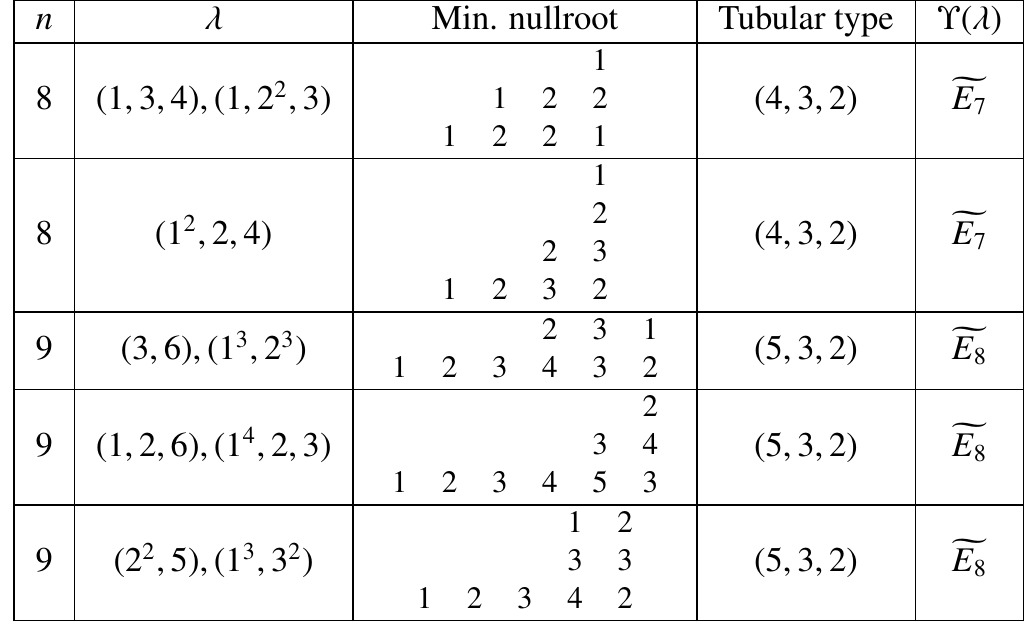}
\end{center}
The table shows every such algebra together with some information. Any tame concealed algebra has a unique one-parameter family of indecomposable modules $X$ with $\End(X) = K$, $\Ext^1(X,X) =K$, and the dimension vector of these modules is the minimal positive nullroot (see \cite{HaVo}). It generates the radical of $q_{\lambda}$. 
The orbit quiver $\Upsilon(\lambda)$ equals the frame of the tilted algebra, that is, the Gabriel quiver of $\B$ and the tubular type is the same as the tubular type of $\B$ \cite{Ri}.  Each tilting module is given by a direct sum of preprojective indecomposables and there is a stable separating tubular $\mathbf{P}_1K$-family of the tubular type of the algebra which separates the preprojective and preinjective component. In these components, all indecomposables come up (up to isomorphism). The regular component equals the $\Hom_{\B}(T,\_)$-translation of the regular component of the Euclidean quiver \cite{Ri}. The simple regular modules are those at the mouths of the tubes; and they have explicitly been calculated for $\widetilde{E_7}$ and $\widetilde{E_8}$ in \cite{DlR2}. Thus, all simple regular modules and all tubes are obtained. Note that by results of Bongartz \cite{Bo6}, orbit closures can be calculated.

\subsection{Non-concealed tame algebras}\label{TameMax}
The non-maximal non-concealed staircase algebras are corresponding to the partitions $\lambda=(2, 3^2)$, $\lambda'=(4,5)$ and $\lambda'^T=(1,2^4)$. Both are Euclidean tilted by Theorem \ref{criterionTameTilted} and Corollary \ref{criterionTiltedSincere} and their orbit type  ($\widetilde{E_7}$ for $\lambda$, $\widetilde{E_8}$ for $\lambda'$)  equals their frame.
 The tilting module is given by a direct sum of preprojective indecomposables plus at least one non-preprojective module, since the algebras are not tame concealed. Each contains a tame concealed subcategory and the radical of the quadratic form is always generated by the induced minimal nullroot, since these algebras are Euclidean tilted and, thus, their quadratic form is isometric to a quadratic form of a quiver. By Lemma \ref{criterionTameTilted} the quadratic form is non-negative and there is no quiver of non-negative quadratic form with $2$-dimensional radical.\\[1ex]
Concerning the maximal tame staircase algebras, we prove tameness without using the list of hypercritical algebras \cite{Un}, now. We know from Theorem \ref{criterionTame} that $\A(\lambda)$ is tame if and only if the quadratic form $q_{\lambda}$ is weakly non-negative. 
\begin{lemma}\label{MaxTameStraightForw3}
Let $\lambda=(3^3)$ and $\lambda'\in\{(5^2), (2^5)\}$, then 
$q_{\lambda}$ and $q_{\lambda'}$ are non-negative and 
\[\rad q_{\lambda} =\left\langle u:=\left(\begin{array}{ccc}
1&1&0\\
1&0&-1\\
0&-1&-1
\end{array} \right),~v:=\left(\begin{array}{ccc}
0&1&1\\
1&2&1\\
1&1&0
\end{array} \right)\right\rangle.\]
\[\rad q_{\lambda'} =\left\langle \left(\begin{array}{ccccc}
2&3&2&0 &-1\\
1&0&-2&-3&-2
\end{array} \right),~\left(\begin{array}{ccccc}
0&1&2	&2&1\\
1&2&2&1&0
\end{array} \right)\right\rangle.\]
In particular,  $\A(\lambda)$ and $\A(\lambda')$ are tame and non-tilted.
\end{lemma}
\begin{proof}
 Let $\lambda'':=(2,3^2)$, then $\A(\lambda'')$ is tilted by Lemma \ref{criterionTiltedSincere}. It is Euclidean tilted, since its preprojective component is of type $\widetilde{E_7}$ ("contains a $\widetilde{E_7}$-slice") and contains all projective indecomposables. By Lemma \ref{criterionTameTilted}, $q_{\lambda''}$ is
non-negative.\\[1ex]
Note that $u$ is orthogonal to every other vector concerning the symmetric form $q_{\lambda}(\_,\_)$; in particular, $q_{\lambda}(u)=0$.\\[1ex] Let $x\in \mathbf{R}^9= \mathbf{R}^{\Q(\lambda)_0}$, then $\tilde{x}:=x-x_1u \in\mathbf{R}^8\cong \mathbf{R}^{\Q(\lambda'')_0}$ and $q_{\lambda}(x)=q_{\lambda''}(\tilde{x})\geq 0$. If $x\in \rad q_{\lambda}$, then $\widetilde{x}\in\rad q_{\lambda''}$,  and, thus, $\widetilde{x}=\xi\cdot v$ for some $\xi$ as has been shown before.\\
 Since $q_{\lambda}$ is non-negative, $\A(\lambda)$ is tame by Lemma \ref{criterionTame}. If $\A(\lambda)$ was tilted, its quadratic form would be isometric to a quadratic form of a quiver. But there is no such non-negative form with $2$-dimensional radical.\\
 The proof for $\lambda'$ can be done analogously.
\end{proof}
Tameness of $(1,4^2)=(2^3,3)^T$ follows from tameness of $(5,5)$ by Lemma \ref{reductioncomp}, since $q_{(1,4^2)} = q_{(4,5)}$. Note that the radical of the quadratic form $q_{\lambda}$  for $\lambda=(1,4,4)$ is $1$-dimensional, since the algebra $\A(4,5)$ is tilted. The orbit type of the algebras $\A(1,4^2)$, $\A(2^3,3)$ and $\A(3^3)$ is $\widetilde{E_7}$ and the orbit type of $\A(5^2)$ and $\A(2^5)$ is $\widetilde{E_8}$. \\[1ex]
By Lemma \ref{criterionTameTilted} and Lemma \ref{MaxTameStraightForw3}, we  have proved that the Tits form of every tame algebra $\A(\lambda)$ is non-negative.
 
\begin{lemma}\label{finiteGrowth}
Let $\A(\lambda)$ be tame, then it has finite growth.
\end{lemma}
\begin{proof}
Let $\A(\lambda)$ be a tame staircase algebras for $\lambda\notin\{(3^3),(2^5),(5^2)\}$. Then $\rad q_{\lambda}$ is $1$-dimensional whence Proposition \ref{criterionFinGrowth} yields the claim. Let $\lambda\in\{(3^3),(2^5),(5^2)\}$, then we know explicit generators of $\rad q_{\lambda}$ from Lemma \ref{MaxTameStraightForw3}. Clearly, $\corank^0 q_{\lambda}\leq 1$, such that the claim follows from Proposition \ref{criterionFinGrowth}, as well.
\end{proof}
\subsection{Minimal wild cases}\label{minWildCases}
The minimal wild cases can be found in the graphic in Appendix \ref{hierarchy}. Some of them are hypercritical (see \cite{Un}), namely $(1,3,5)$, $(1^2,2,5)$, $(3,7)$, $(2^2,6)$, $(1,2,7)$ and their transposes. Furthermore, all orbit quivers of the minimal wild cases are of extended Euclidean type $\widetilde{\widetilde{E_7}}$ or $\widetilde{\widetilde{E_8}}$.\\[1ex]
Let $n\geq 2$, then concrete $n$-parameter families can be produced nicely. Let $\A(\lambda)$ be a minimal wild staircase algebra and let $\A(\lambda')$ be a convex subcategory which is tame concealed and is obtained from $\A(\lambda)$ by cancelling a source vertex $x$ in which one arrow $\alpha: x\rightarrow y$ starts. \\[1ex]
Let $M'$ be a preprojective indecomposable $\A(\lambda')$-representation, such that $\dim_K M'_y=n+1$. Then the $\A(\lambda)$-representations defined by $M_i =M'_i$ if $i\neq x$ and $M_x=K$ together with the induced maps of $M'$ and the embedding \[M_{\alpha}: K\xrightarrow{\left(\begin{array}{c}
a_1\\
\vdots\\
a_{n+1}\\
\end{array}\right)} M_x\]
gives a $\mathbf{P}^{n}$-family of pairwise non-isomorphic indecomposables.
\begin{remark}\label{explicit2Param}
Explicit $2$-parameter families are induced by these dimension vectors:
\begin{tabular}{|c|c|c|c|c|}
\hline
 \begin{footnotesize}\begin{tikzpicture}
\matrix (m) [matrix of math nodes, row sep=0.2em,
column sep=0.1em, text height=0.79ex, text depth=0.1ex]
{    & 2& \textbf{3}\\  
 2 & 5& 7\\
  \textbf{3}&  7&  \\
 };
\end{tikzpicture}  \end{footnotesize}& 

\begin{footnotesize}\begin{tikzpicture}
\matrix (m) [matrix of math nodes, row sep=0.2em,
column sep=0.1em, text height=0.79ex, text depth=0.1ex]
{&&& \textbf{3}\\
 &2& 5& 6 \\
\textbf{3}& 6 &  7& 4\\
 };
\end{tikzpicture}  \end{footnotesize}

& \begin{footnotesize}\begin{tikzpicture}
\matrix (m) [matrix of math nodes, row sep=0.2em,
column sep=0.1em, text height=0.79ex, text depth=0.1ex]
{ & && 5& 8& 6& \\
\textbf{3}& 6& 9& 12&  10& 4\\};
\end{tikzpicture}  \end{footnotesize} 
& \begin{footnotesize}\begin{tikzpicture}
\matrix (m) [matrix of math nodes, row sep=0.2em,
column sep=0.1em, text height=0.79ex, text depth=0.1ex]
{&&&  \textbf{3}\\
 &&& 6\\
 && 5& 9\\
\textbf{3} & 6 & 9& 7\\
 };
\end{tikzpicture}  \end{footnotesize}&

 \begin{footnotesize}\begin{tikzpicture}
\matrix (m) [matrix of math nodes, row sep=0.2em,
column sep=0.1em, text height=0.79ex, text depth=0.1ex]
{ & && & & 6& \\
& && & 8& 12& \\
\textbf{3}& 6& 9& 12&  15& 10\\};
\end{tikzpicture}  \end{footnotesize}  \\\hline
\end{tabular}
(the marked entries show, where an embedding of $K$ leads to a $2$-parameter family). This way, for each wild staircase algebra, a $2$-parameter family can be constructed.
\end{remark}
\section{Results for graded nilpotent pairs}\label{SectGradedNilp}
We obtain certain results for graded nilpotent pairs, which immediately follow from our examinations of staircase algebras. Let $\lambda$ be a partition and 
let $V$ be a bigraded finite-dimensional vector space of shape $\lambda$.

\begin{theorem}\label{gradedLambda}
The number of $\lambda$-graded nilpotent pairs (up to Levi-base change) is finite if and only if 
\begin{enumerate}
 \item $\lambda\in\{(n),~ (1^k,n-k),~ (2,n-2),~ (1^{n-4},2^2)\}$ for some $k\leq n$,
\item  $n\leq 8$ and $\lambda\notin\{(1,3,4),~ (2,3^2),~ (1,2^2,3),~ (1^2,2,4)\}$.
\end{enumerate}
\end{theorem}
\begin{proof}
Follows from Theorem \ref{gradedCriterionLambda} and Theorem \ref{reptype}: The staircase algebra $\A(\lambda)$ has finite representation type if and only if the number of isomorphism classes of representations is finite. The latter correspond to $\lambda$-graded nilpotent pairs (up to Levi-base change) bijectively.
\end{proof}
If $\A(\lambda)$ is tame, then there are only one-parameter families of $\lambda$-graded nilpotent pairs (up to Levi-base change).
\begin{lemma}\label{gradedTame}
Assume that $\lambda\in\{(1,3,4),~ (1,2^2,3),~ (1^2,2,4),~(3,6),~(1^3,2^3),~(1,2,6),$\\
$(1^4,2,3),~(2^2,5),~(1^3,3^2)\}$. Then there are only finitely many graded nilpotent pairs on $V$ (modulo base change in the homogeneous components) if and only if $\underline{\dim}V$ does not contain a minimal nullroot as in \ref{TC}.
\end{lemma}

\begin{proof}
If $\A(\lambda)$ is tame concealed, then the number of isomorphism classes of a fixed dimension vector $\underline{d}$ is infinite if and only if $\underline{d}$ does not contain a minimal nullroot. Thus, the claim follows from Theorem \ref{reptype}.
\end{proof}

\begin{lemma}\label{gradedInfinite}
If $\underline{\dim}V$ contains a minimal nullroot (see Section \ref{SectCasestudy} for the explicit list), then the number of graded nilpotent pairs on $V$ is (up to isomorphism) infinite.
\end{lemma}

\appendix
 \section{Auslander-Reiten quivers}\label{appendix}
 
 \subsection[The case (3,4)]{The case $\lambda=(3,4)$}\label{ark34}
\begin{center}
\includegraphics[trim=122 520 155 130, clip,width=320pt]{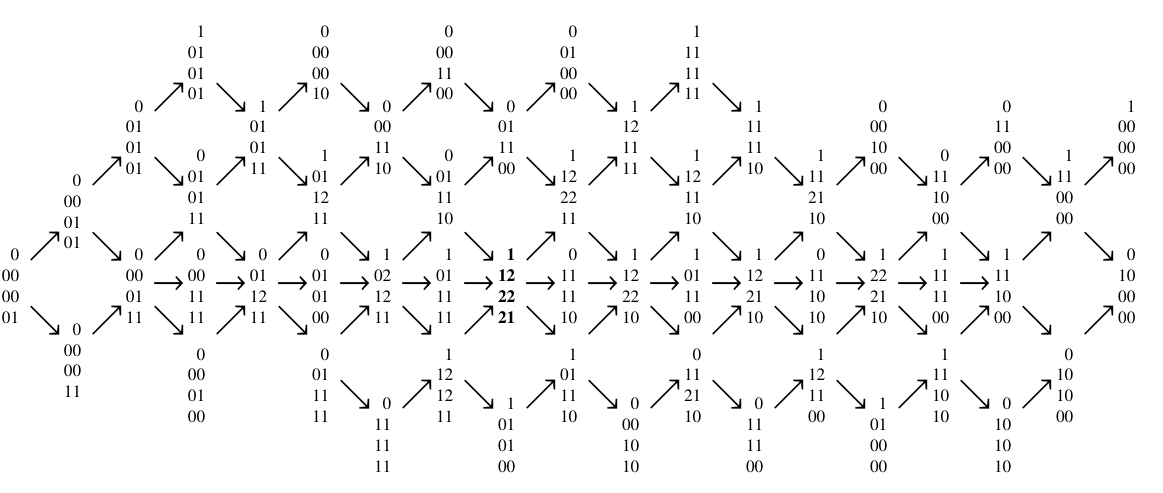}
\end{center}
 
\subsection[The case (2,6)]{The case $\lambda=(2,6)$}\label{ark26} Generalizes easily to $\lambda=(1,...,1,2,2)$ and $\lambda=(2,n-2)$. 
\begin{center}
\includegraphics[trim=112 440 135 130, clip,width=310pt, height=140pt]{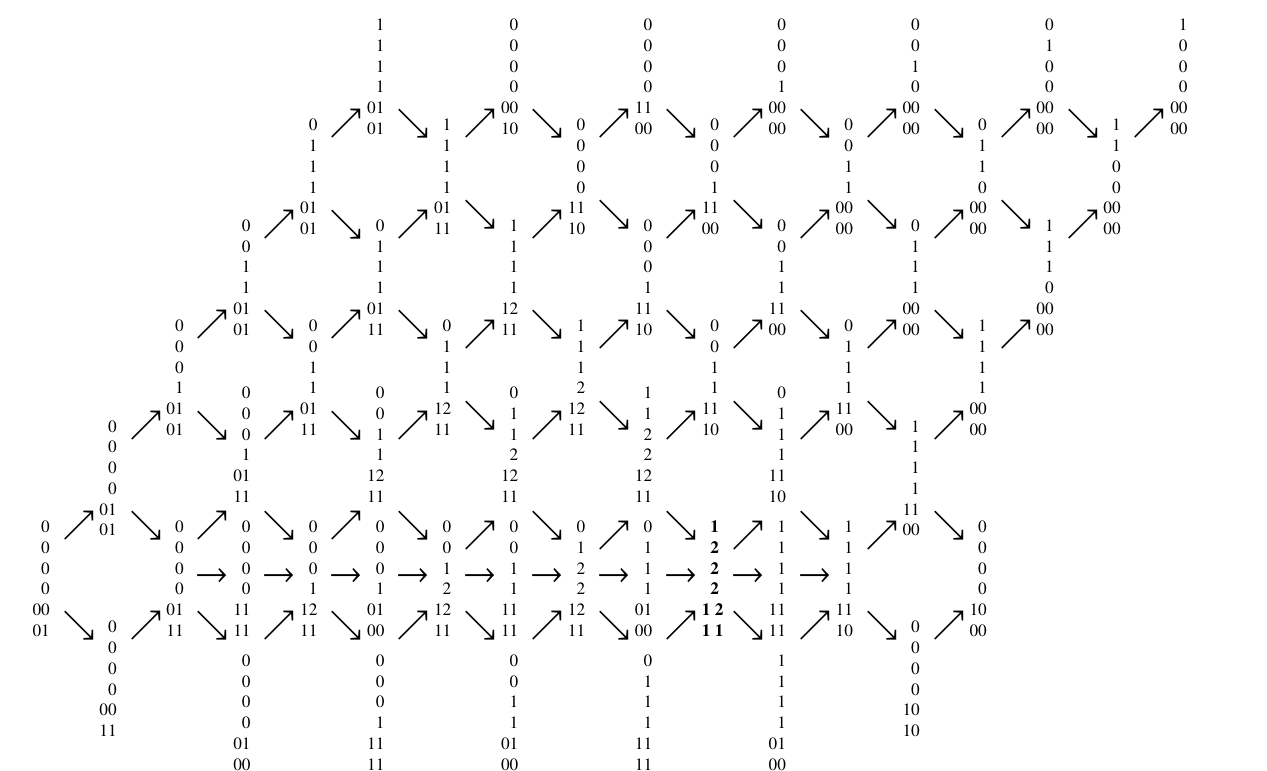}
\end{center}

\subsection[The case (1,2,3)]{The case $\lambda=(1,2,3)$}\label{ark123}
\begin{center}
\includegraphics[trim=118 555 235 125, clip,width=250pt,height=90pt]{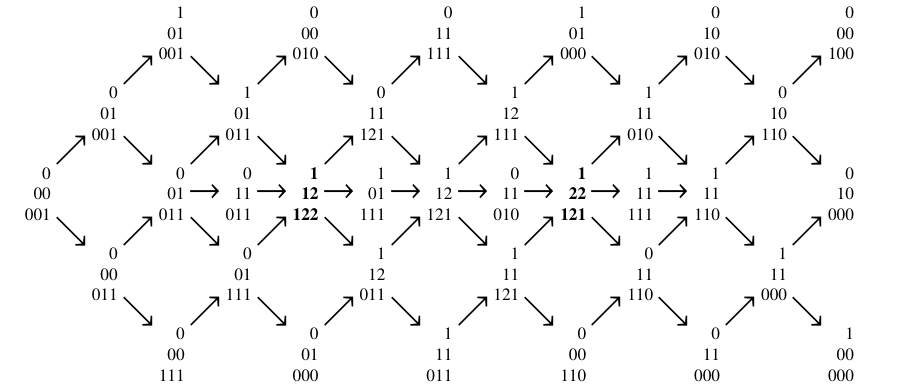}
\end{center}

\subsection[The case (2,2,2)]{The case $\lambda=(2,2,2)$}\label{ark222}

\begin{center}
\includegraphics[trim=138 585 234 102, clip,width=200pt, height=90pt]{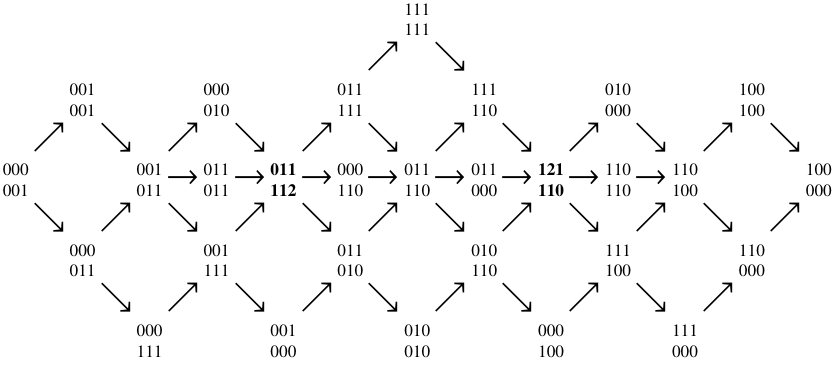}
\end{center}

\subsection[The case (1,3,3)]{The case $\lambda=(1,3,3)$}\label{ark133}
\begin{center}
\includegraphics[trim=118 520 100 130, clip,width=320pt, height=110pt]{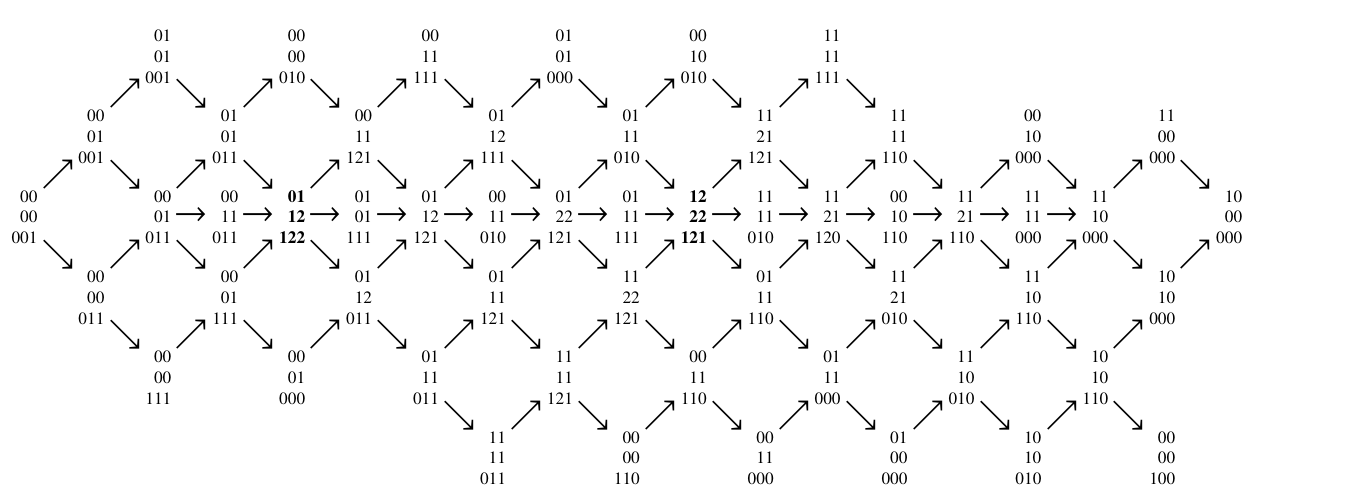}
\end{center}

\subsection[The case (3,5)]{The case $\lambda=(3,5)$}\label{ark35}

\begin{center}
\includegraphics[trim=8 480 0 135, clip,width=380pt, height=100pt]{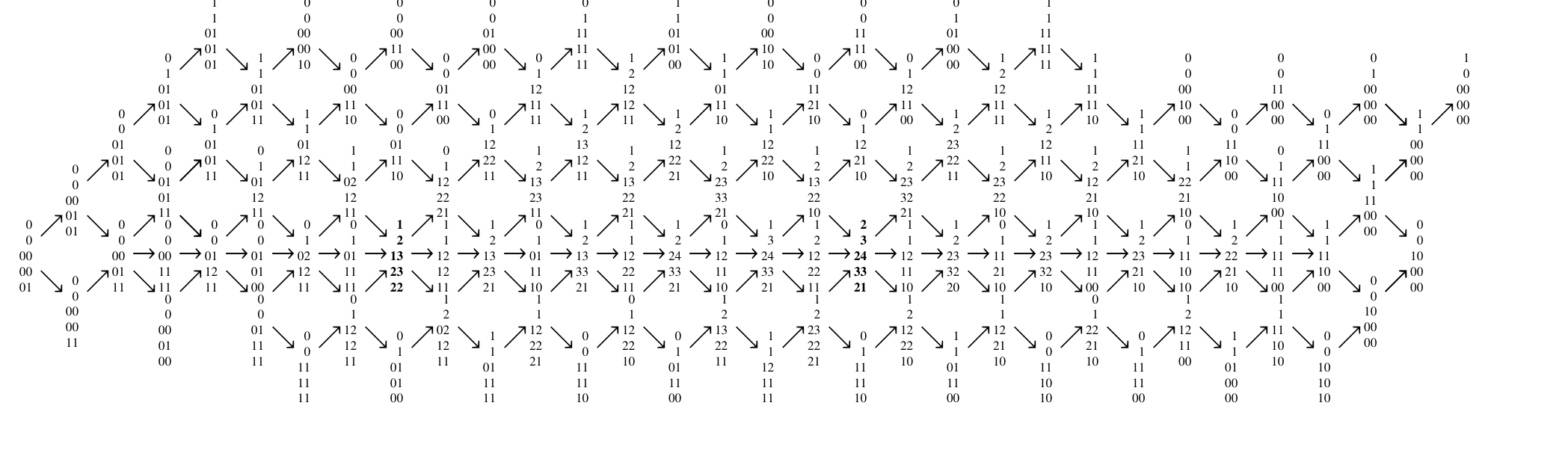}
\end{center}

\subsection[The case (1,2,4)]{The case $\lambda=(1,2,4)$}\label{ark124}
\begin{center}
\includegraphics[trim=118 520 135 130, clip,width=320pt, height=89pt]{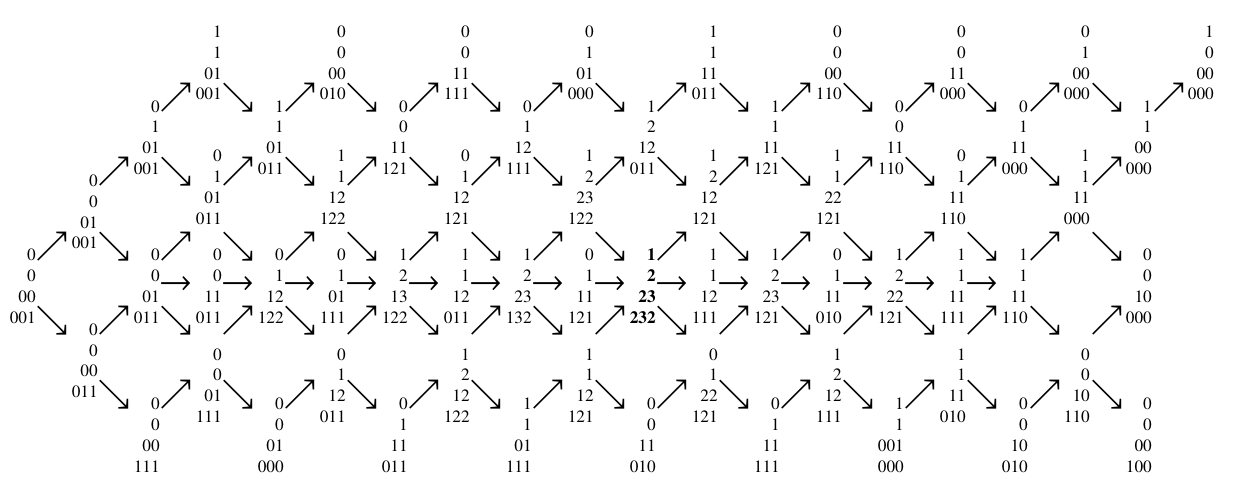}
\end{center}

\subsection[The case (1,2,5)]{The case $\lambda=(1,2,5)$}\label{ark125}
\begin{center}
\includegraphics[trim=8 480 0 135,width=380pt, height=98pt]{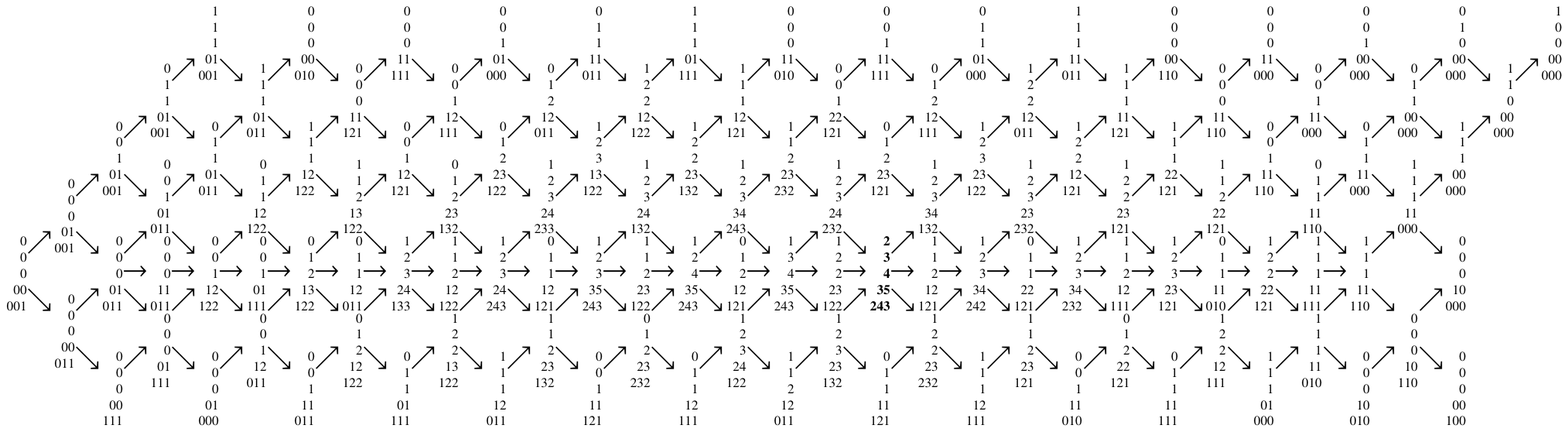}
\end{center}

\subsection[The case (2,2,4)]{The case $\lambda=(2^2,4)$}\label{ark224}

\begin{center}
\includegraphics[trim=8 480 0 139,clip,width=380pt, height=98pt]{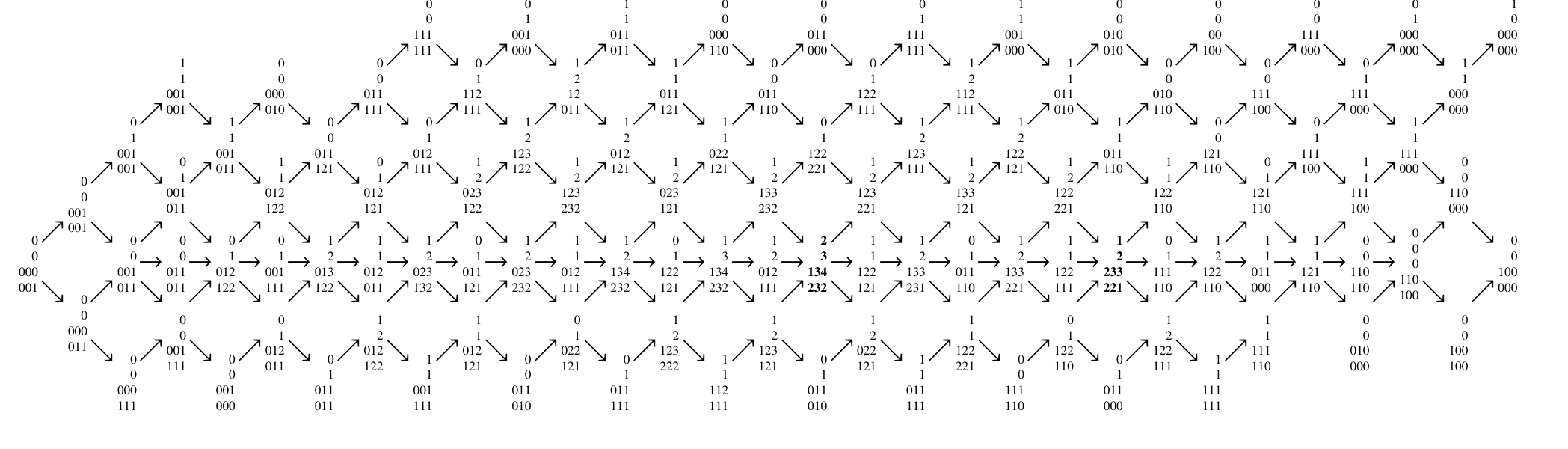}
\end{center}

\subsection[The case (2,2,2,2)]{$\lambda = (2^4)$ or $\lambda = (4^2)$}\label{ark2222}
\begin{center}
\includegraphics[trim=58 480 0 147, clip,width=380pt, height=98pt]{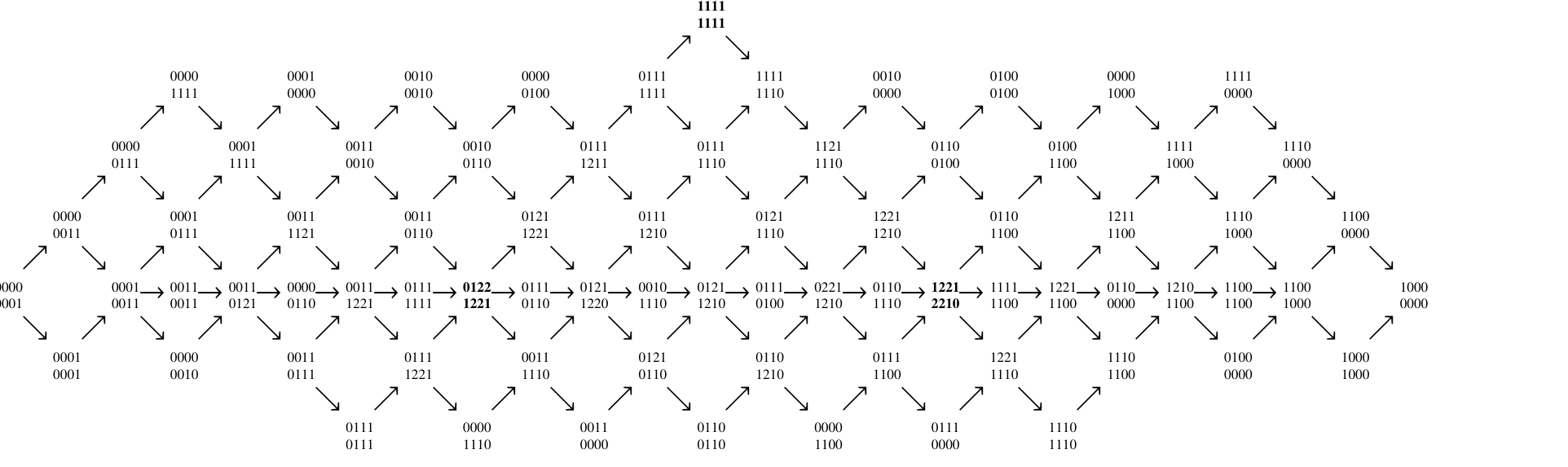}
\end{center}

 \section{Hierarchy of representation types}\label{hierarchy}
\begin{center}
\includegraphics[trim=133 155 97 125, clip,width=380pt]{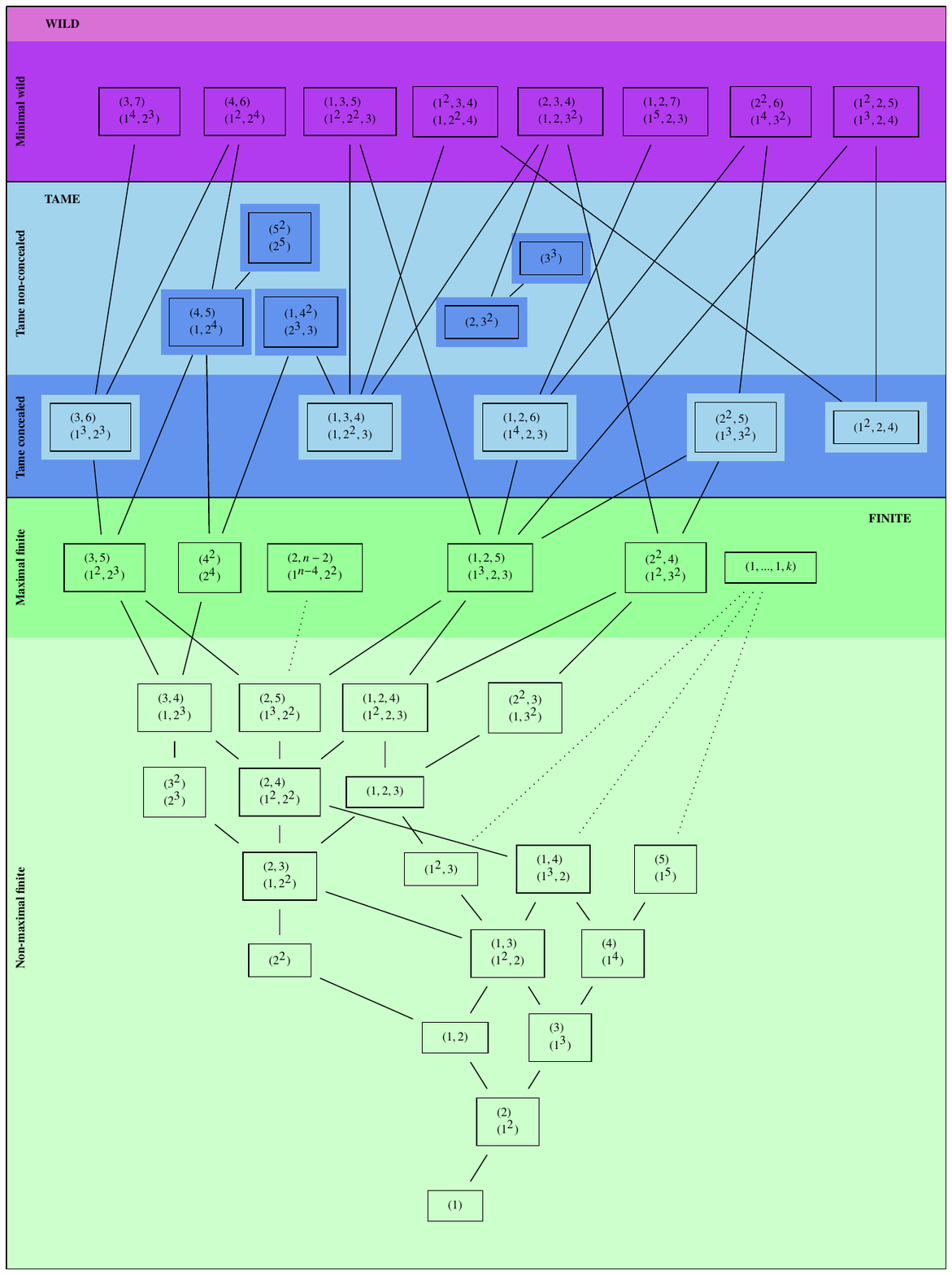}
\end{center}

\end{document}